\documentclass{amsart}
\usepackage[utf8]{inputenc}
\usepackage{amsmath,amsthm,amssymb}
\usepackage[all,cmtip]{xy}
\usepackage{color}
\usepackage{tikz}
\usetikzlibrary{matrix,arrows}

\newtheorem{theorem}{Theorem}[section]
\newtheorem{lemma}[theorem]{Lemma}
\newtheorem{proposition}[theorem]{Proposition}

\newtheorem{claim}[theorem]{Claim}

\theoremstyle{definition}
\newtheorem{definition}[theorem]{Definition}

\newtheorem{problem}[theorem]{Problem}
\newtheorem{example}[theorem]{Example}

\newtheorem{definitions and remarks}[theorem]{Definitions and Remarks}

\theoremstyle{remark}

\newtheorem{remark}[theorem]{Remark}
\newtheorem{notation}[theorem]{Notation}

\numberwithin{equation}{section} 

\newcommand{\umon}[1]{\pmb{u}^{\pmb{#1}}}
\newcommand{\wmon}[1]{\pmb{w}^{\pmb{#1}}}
\newcommand{\xmon}[1]{\pmb{x}^{\pmb{#1}}}

\newcommand{\sm}[1]{{\scriptstyle (#1)}}

\newcommand{\RMIauthor}[1]{{\sc #1:}} 
\newcommand{\RMIpaper}[1]{{\rm #1.}} 
\newcommand{\RMIbook}[1]{{\sl #1.}} 
\newcommand{\RMIjournal}[1]{{\sl #1}}

\title[Local monomialization of first integrals]{Local monomialization of a system of first integrals of Darboux type}
\author[A.~Belotto]{Andr\'e Belotto da Silva}
\address{University of Toronto, Department of Mathematics, 40 St. George Street,
Toronto, ON, Canada M5S 2E4}
\email[A.~Belotto]{{andrebelotto@gmail.com}}

\subjclass{Primary: 32S45, 32S65. Secondary: 34C08, 34C20}

\begin{document}

\maketitle

\section*{Abstract}
Given a real- or complex-analytic singular foliation $\theta$ with $n$ first integrals of {meromorphic or} Darboux type $(f_1,\dots,f_n)${,} we prove that there exists a local monomialization of the first integrals. In particular, if $\theta$ is generated by the $n$ first integrals, we prove the existence of a local reduction of singularities of $\theta$ to monomial singularities.

\setcounter{tocdepth}{3}

\section{Introduction}

The subject of this article is \emph{reduction of singularities} of singular foliations, a classical problem which have interested mathematicians since the beginning of twentieth century \cite[1901]{BE}. The best results to date are valid only in low-dimensions; e.g. resolution of singularities of foliations in dimension two (Bendixson and Seidenberg \cite{BE,Sei}), and dimension three (Cano \cite[2004]{Can}, Panazzolo \cite[2006]{Pan} and McQuillan and Pannazolo \cite[2014]{McP} - see also Cano, Roche and Spivakovsky \cite[2015]{CRS}). In arbitrary dimensions, Camacho, Cano and Sad have proven a resolution of singularities of vector fields under the additional hypothesis that all singularities are absolutely isolated \cite[1989]{CS}. {In this paper we are interested in providing a monomialization of first integrals (of meromorphic or Darboux type). In particular, we provide a local reduction of singularities of completely integrable foliations in arbitrary dimensions.} One of the motivations of this problem is the study of pseudo-abelian integrals \cite{Pavao,Nov}.

Let $M$ be a complex- or real-analytic manifold (i.e., the base field $\mathbb{K}$ is $\mathbb{R}$ or $\mathbb{C}$) and $\theta$ be an \emph{involutive singular distribution} (i.e., a coherent sub sheaf of the sheaf of vector fields over $M$, denoted by $Der_M$, such that for each point $p$ in $M$ the stalk $\theta \cdot \mathcal{O}_p$ is closed under the Lie bracket operation). Note that $\theta$ generates a singular foliation over $M$ (by the Stefan-Sussmann Theorem \cite{Ste,Suss}). 

Denote by $\mathcal{K}$ a sub-field of $\mathbb{K}$. We say that $\theta$ has $n$ first integrals of $\mathcal{K}$\emph{-Darboux type} (without an exponential factor) at a point $p \in M$ {if the foliation generated by $\theta$ is tangent to the leaves of $n$ meromorphic $1$-form germs
\[
\omega_i:= \sum_{j} k_{i,j} \frac{d g_{i,j}}{g_{i,j}}, \quad k_{i,j} \in \mathcal{K} \text{ and } g_{i,j} \in \mathcal{O}_p, \text{ for } i=1, \ldots, n
\]
such that $\omega_1 \wedge \ldots \wedge \omega_n \not\equiv 0$. Equivalently, there exists $n$ (complex multi-valued) function germs $f_i = \prod g_{i,j}^{k_{i,j}}$ such that} $df_1 \wedge \cdots \wedge df_n \not\equiv 0$ and $\partial(f_i) \equiv 0$ for all derivations $\partial$ in the stalk ${\theta \cdot \mathcal{O}_p}$.

{
\begin{remark}\label{rk:DarbouxFI}
In many references, Darboux first integrals (also called ``generalized Darboux" first integrals) have an exponential factor, that is, the first integrals have the form $f = exp\left( \frac{\phi}{\psi}\right)\prod g_{j}^{k_{j}} $ where $\phi$ and $\psi$ are analytic germs (compare \cite{Pavao,Darboux2,Nov}). In this work, we always assume that $\psi=1$ (c.f. \cite{Darboux2}).  
\end{remark}
}

We address the following problem:

\begin{problem}
Suppose that $\theta$ has $n$ first integrals of $\mathcal{K}$-Darboux type at a point $p \in M$. The problem consists in finding a bimeromorphic and proper morphism $\sigma: \widetilde{M} \to M$ {(and a simple normal crossing divisor $\widetilde{E}\subset \widetilde{M}$ such that $\sigma$ is an isomorphism outside of $\widetilde{E}$)} such that the transform $\widetilde{\theta}$ of $\theta$ has $n$ \emph{monomial} first integrals \emph{adapted} to {$\widetilde{E}$}, {More precisely,} at every point $p$ in $\widetilde{M}$, there exists a coordinate system $\pmb{x} = (x_1,\ldots,x_m)$ and $n$ first integrals $\pmb{x}^{\pmb{\alpha_1}}, \ldots, \pmb{x}^{\pmb{\alpha_n}}$, where $\pmb{x}^{\pmb{\alpha_i}} := x_1^{\alpha_{i,1}} \cdots x_m^{\alpha_{i,m}}$ such that:
\begin{enumerate}
\item $supp \, \widetilde{E} = \{x_1 \cdots x_l=0\}$ for some $l$;
\item the multiindexes $\pmb{\alpha_1},\ldots,\pmb{\alpha_n} \in \mathcal{K}^m$ are linearly independent over $\mathbb{K}$.
\end{enumerate}
\label{prb:main}
\end{problem}

The above problem can be strengthened by asking that $\sigma$ be a composite of blowings-up with smooth admissible centers (\emph{admissible} means that each center of blowing-up has only normal crossings with the exceptional divisor). In this case we can write $\sigma: (\widetilde{M},\widetilde{E}) \to (M,E)$, where $\widetilde{E}$ is the union of the strict transform of $E$ with the exceptional divisors of each blowing up.

A positive solution of problem \ref{prb:main} {seems to have} applications to pseudo-abelian integrals, an important technique used to estimate the number of limit cycles which bifurcate from a {Darboux} planar vector-field \cite{Pavao,Nov}. It has been suggested in \cite{Aymen} (and {in a} personal communication from Pavao Mardesic) that a positive answer to Problem \ref{prb:main}, mixed with the techniques from \cite{Pavao,AymenT}, could lead to new results about non-generic pseudo-abelian integrals.

In the present work we present a \emph{local} reduction of first integrals (i.e., a local solution of problem \ref{prb:main}). By local, we mean that we accept admissible \emph{local blowings up}, i.e. the composition of an admissible blowing-up with an open immersion (e.g. a chart of a blowing up). More precisely:

\begin{theorem}\label{thm:main}
Let $M$ be a non-singular analytic manifold, $E$ be a simple normal crossing divisor on $M$, $p$ be a point of $M$ and $\theta$ be a singular distribution with $n$ first integrals of $\mathcal{K}$-Darboux type $(f_1, \ldots, f_n)$ over $p$ such that $d f_1 \wedge ... \wedge d f_n \not\equiv 0$. Then there exists a finite collection of morphisms $\tau_i : (M_i,E_i) \to (M,E)$ such that:
\begin{enumerate}
 \item The morphism $\tau_i$ is a finite composition of admissible local blowing-ups;
 \item There exists a compact set $K_i \subset M_i$ such that $\bigcup \Phi_i(K_i)$ is a compact neighborhood of $p$;
 \item The strict transform $\theta_i$ of the singular distribution $\theta$ (by $\tau_i$) has $n$ monomial first integrals adapted to $E_i$.  Furthermore, if $\mathcal{K} = \mathbb{Q}$, then the $n$ monomial first integrals are analytic.
\end{enumerate}
In other words, Problem \ref{prb:main} admits a positive solution by local blowings-up.
\end{theorem}

\begin{example}
Consider the case when $dim \, M=3$ and $\theta$ is generated by two analytic first integrals $f_1$ and $f_2$. Then, Theorem \ref{thm:main} states that the singularities of $\theta$ can be reduced to $\mathbb{Q}$-\emph{monomial} singularities (see definition \ref{def:MSD}) which are locally equal to one of the following three forms:
\begin{enumerate}
\item $\theta$ is generated by the regular vector field $\partial_x$, i.e. $(f_1,f_2)=(y,z)$.
\item $\theta$ is generated by the linear vector field $\alpha_1 y\partial_y - \alpha_2 z\partial_z$ with $(\alpha_1,\alpha_2)\in \mathbb{N}^2$, i.e. $(f_1,f_2)=(x,y^{\alpha_2}z^{\alpha_1})$.
\item $\theta$ is generated by the linear vector field $\alpha_1 x\partial_x - \alpha_2y\partial_y - \alpha_3z\partial_z$  with $(\alpha_1,\alpha_2,\alpha_3)\in \mathbb{N}^3$, i.e. $(f_1,f_2)=(x^{\alpha_2}y^{\alpha_1},x^{\alpha_3}z^{\alpha_1})$.
\end{enumerate}
\label{ex:1}
\end{example}

\subsection{Simultaneous resolution of singularities and monomialization of morphisms}

The originality of this result comes from the fact that Problem \ref{prb:main} does not follow {in an evident way from} resolution of singularities of varieties (e.g. \cite{BM,Vil4}). For the following discussion, let us assume that the $n$ first integrals $f_1, \ldots, f_n$ of $\theta$ are all analytic. In this case, one could try to solve Problem \ref{prb:main} by \emph{simultaneous resolution of singularities}, { that is,} by principalization of the ideal generated by $\Pi_{i=1}^n f_i \Pi_{i<j}(f_i-f_j)$. At the end of the process, the pull-back of each germ $f_i$ is locally given by a monomial in the exceptional divisor times a unit. But this is not enough, as we can see in the following example:
\begin{example}
\label{ex:main}
In a three dimensional manifold, consider two first integrals
\[
f_1 = x^2+y^2 , \quad f_2 = x^4+y^4+ y^2z^2+z^4
\] 
and the ideal $I = f_1f_2(f_1-f_2)$. After blowing-up the origin consider the origin of the $x$-chart with coordinate system $(x,y,z)= (u,uv,uw)$. In this chart:
\[
f_1\circ \sigma = u^2(1+v^2) , \quad f_2\circ \sigma = u^4(1+v^4+ v^2w^2+w^4)
\]
and the pulled-back ideal $I^{\ast}$ is principal. Nevertheless, we can not absorb the units of $f_1\circ \sigma$ and $f_2\circ \sigma$ into the monomials $u^2$ and $u^4$ simultaneously. In particular, the foliation generated by these first integrals is not topologically equivalent to one of the three cases in example \ref{ex:1}.
\end{example}

To solve Problem \ref{prb:main}, we need to ``monomialize the sub-ring" $(f_1,\ldots,f_n)$ instead of the ideal. In order to do so, we combine techniques developed for resolution of singularities \emph{subordinate to foliations} \cite{Bel2} and \emph{monomialization of morphisms} \cite{Cut3}. It is worth mentioning that, if the first integrals are all analytic, then {the preprint of Cutkosky on} local monomialization of analytic morphisms \cite{CutA} {provides a different proof}.

\subsection{Overview of the proof}
The proof follows by induction on the leaf dimension of an auxiliary singular distribution $\omega$. More precisely, in subsection \ref{ssec:FSR}, we introduce the notion of \emph{foliated $\mathcal{K}$-Darboux data} $(M,\omega,\mathcal{D},E)$, where $\omega$ is a singular distribution and $\mathcal{D}$ is given by the {(complex multi-valued)} first integrals $(f_1,\ldots,f_n)$. The singular distribution $\omega$ is an auxiliary singular distribution that contains $\theta$ and is $\mathcal{K}$-\emph{monomial}, i.e. it (almost) satisfies the thesis of Problem \ref{prb:main} (see definition \ref{def:MSD} and Lemma \ref{lem:Fi}). In particular, if the codimension of $\omega$ is ${n}$, then $\omega$ (and consequently $\theta$) have $n$ monomial first integrals. {Furthermore}, if $\mathcal{K} \subset \mathbb{R}$, we {make extra blowings-up} in order to guarantee that the monomial first integrals do not have poles (Lemma \ref{lem:Final}).

In order to decrease the leaf dimension of $\omega$, we prove the existence of a ``monomialization" {of (a codimension one foliation associated to)} $\mathcal{D}$ which preserves the class of monomial singularities {of $\omega$} (Theorem \ref{thm:main2}). More precisely, apart from local blowings up, we can assume that there exists $f\in \mathcal{D}$ such that $f$ is a monomial which is \emph{not} a first integral of $\omega$ (see Lemma \ref{lem:invariantzerodone}). This allow us to construct a new foliated Darboux data $(M,\omega',\mathcal{D},E)$ where, at each point $p \in M$, the stalk $\omega'_p$ is given by $\{\partial \in \omega; \partial(f)\equiv 0  \}$. The codimension of $\omega'$ is strictly bigger than the codimension of $\omega$ and, moreover, we prove that $\omega'$ is also a $\mathcal{K}$-monomial singular distribution (Lemma \ref{lem:invariantzerodone}). So, we start with $\omega=Der_M(-logE)$ and we repeat the process until the codimension of $\omega$ is ${n}$ (see details in section \ref{sec:Overview1}).

The proof of Theorem \ref{thm:main2} is technically the hardest part of the paper. As a first step, we need to guarantee that the transform of $\omega$ under the necessary blowings up are going to be $\mathcal{K}$-monomial. This is not true if we are not careful with the blowings up we perform:

\begin{example}
\label{ex:2}
Consider a three dimensional regular variety and the singular distribution $\omega = (\partial_x + x\partial_z)$, which is a regular (and, therefore, monomial) singular distribution. Let us consider the blowing-up with center $\mathcal{C}=V(y,z)$. In the $y$-chart $(x,y,z)=(u,v,vw)$, the transform of $\omega$ is generated by $v\partial_{u} + u\partial_{w} $. Note that the linear part of this vector field is nilpotent and, therefore, the strict transform of $\omega$ is not monomial (nor log-canonical).
\end{example}

The example {suggests} that we {may want} to impose some restriction to the centers of blowing up. More precisely, in section \ref{sec:tadm}, we recall the notion of ${\omega}$\textit{-admissible} blowings-up (see definition \ref{def:tadm}), which was first introduced in \cite{Bel2,BeloT}. This kind of blowings-up preserve the class of $\mathcal{K}$-monomial singularities (Proposition \ref{prop:AdmCenter}). Further results about ${\omega}$-admissible blowings-up which are necessary (Proposition \ref{prop:AdmCenter} and Theorems \ref{thm:RinvI} and \ref{thm:RI}) are enunciated in section \ref{sec:tadm} and, for shortness, we refer to \cite{Bel2} for their proofs. Finally, we only need to find a ``monomialization" {of (a codimension one foliation associated to)} $\mathcal{D}$ by ${\omega}$-admissible blowings-up in order to prove Theorem \ref{thm:main2}.

The proof of Theorem \ref{thm:main2} now follows by induction on an invariant $\nu$ associated with a foliated Darboux data (see definition \ref{def:invariant}). The induction has three main steps (as in \cite{AB,Bel2}, c.f \cite{Cut3}) which are presented in details in section \ref{sec:Overview} and proved in sections \ref{sec:Infinte}, \ref{sec:prep} and \ref{sec:drop}. Technically, the main difficulty is that the invariant $\nu$ has no \emph{surface of maximal contact} associated to it, i.e. we can not use the standard ideas of Hironaka in order to argue by induction on the dimension of $M$. In order to deal with this issue, we blow-up centers which are \emph{not} necessarily contained in the maximal locus of $\nu$ (as in \cite{Bel2,Cut3}). This allow us to emulate the existence of a surface of maximal contact and to control the transforms of the foliated Darboux data. Nevertheless, we need to choose an special direction at each step, which means that the algorithm is only local instead of global. Finally, it is worth remarking that one can globalize the algorithm in dimension three (c.f. \cite{AB,Cut3}).

\section{Notation and background}
 
\subsection{Singular distributions} Let $Der_M$ denote the sheaf of analytic vector fields on $M$, i.e. the sheaf of analytic sections of $TM$. An {\em involutive singular distribution} is a coherent {subsheaf} ${\omega}$ of $Der_M$ such that, for each point $p$ in $M$, the stalk ${\omega}_p:={\omega} \cdot \mathcal{O}_p$ is closed under the Lie bracket operation. Consider the quotient sheaf $Q = Der_M/ {\omega}$.  The {\em singular set} of ${\omega}$ is defined by the closed analytic subset $S({\omega}) = \{p \in M : Q_p \text{ is not a free $\mathcal{O}_p$ module}\}$. A singular distribution ${\omega}$ is called \textit{regular} if $S({\omega})=\emptyset$. On $M \setminus S({\omega})$ there exists an unique analytic subbundle $L$ of $TM |_{ M\setminus S({\omega})}$ such that ${\omega}$ is the sheaf of analytic sections of $L$. We assume that the dimension of the $\mathbb{K}$ vector space $L_p$ is the same for all points $p$ in $M \setminus S$ (this always holds if $M$ is connected). It is called the {\em leaf dimension} of ${\omega}$ and denoted by $d$. In this case ${\omega}$ is called an involutive \textit{$d$-singular distribution}.\\
\\
Let $Der_M(-logE)$ denote the coherent {subsheaf} of $Der_M$ {given by} all derivations tangent to $E$, that is, {for each point $p \in M$}, a derivation $\partial$ is in $Der_M(-logE)\cdot \mathcal{O}_p$ if and only if {$\partial[\mathcal{I}_E] \subset \mathcal{I}_E$, where $\mathcal{I}_E$ is the reduced ideal sheaf whose support is $E$}. A singular distribution ${\omega}$ which is also a sub sheaf of $Der_M(-logE)$ is said to be tangent to $E$.

\subsection{Local blowings-up and {complete} collection of local blowings-up}
\label{ssec:CLB} 
A blowing-up $\sigma: \widetilde{M} \to M$ is said to be \emph{admissible} if the center of blowing-up $\mathcal{C}$ has normal crossings with $E$. In this case, we denote the blowing-up by $\sigma : (\widetilde{M},\widetilde{E}) \to (M,E)$, where $\widetilde{E}$ is the union of the inverse-image of $E$ with the exceptional divisor $F$ of the blowing-up.\\
\\
An admissible \textit{local} blowing-up $\tau: (\widetilde{M},\widetilde{E}) \to (M,E)$ is the composition of an admissible blowing with an open immersion (e.g. a chart of the blowing-up). A \emph{sequence of local blowings-up} is a sequence of morphisms
\[
 \begin{tikzpicture}
  \matrix (m) [matrix of math nodes,row sep=3em,column sep=3em,minimum width=1em]
  {(M_r,E_r) & \cdots & (M_0,E_0)\\};
  \path[-stealth]
    (m-1-1) edge node [above] {$\tau_r$} (m-1-2)
    (m-1-2) edge node [above] {$\tau_1$} (m-1-3);
\end{tikzpicture}
\]
where each morphism is an admissible local blowing-up. A \emph{{complete} collection of local blowings-up $\tau_i$} at a point $p \in M$ is a finite collection of morphisms $\tau_i : (M_i,E_i) \rightarrow (M,E)$ such that:
\begin{enumerate}
 \item The morphism $\tau_i$ is a finite composition of admissible local blowing-ups. 
 \item There exists compact {sets} $K_i \subset M_i$ such that $\bigcup \tau_i(K_i)$ is a compact neighborhood of $p$.
\end{enumerate}

\subsection{Blowing-up of a singular distributions}
A \textit{foliated manifold} is the triple $(M,{\omega},E)$ where $M$ is a real- or complex-analytic regular manifold, $E$ is a \emph{simple normal crossing divisor} (i.e., an ordered collection $E = (E^{(1)},...,E^{(l)})$, where each $E^{(i)}$ is a smooth divisor on $M$ such that $\sum_i E^{(i)}$ is a reduced divisor with simple normal crossings) and ${\omega}$ is an involutive singular distribution tangent to $E$ (i.e ${\omega} \subset Der_M(-logE)$).\\
\\
Given an admissible blowing-up $\sigma : (\widetilde{M},\widetilde{E}) \to (M,E)$, we denote by $\widetilde{{\omega}}$ the intersection of {the transform} of ${\omega}$ with $Der_{\widetilde{M}}(-log\widetilde{E})$. In particular, this guarantees that $(\widetilde{M},\widetilde{{\omega}},\widetilde{E})$ is a foliated ideal sheaf and we can write $\sigma : (\widetilde{M},\widetilde{{\omega}},\widetilde{E}) \to (M,{\omega},E)$.

\subsection{Foliated Darboux-data}
\label{ssec:FSR}

We recall that the base field $\mathbb{K} = \mathbb{R}$ or $\mathbb{C}$. Let $\mathcal{K}$ be a subfield {of $\mathbb{K}$}.
\begin{definition}
A \textit{foliated $\mathcal{K}$-Darboux data} is a quadruple $(M,{\omega},\mathcal{D},E)$ where $\mathcal{D}$ is a fixed collection of $n$ \emph{complex} {multi-}valued functions $(f_1,\ldots, f_n)$ of $\mathcal{K}$-Darboux type, that is
\[
f_{i} = \prod_j  g_{i,j}^{k_{i,j}}, \quad {k_{i,j} \in \mathcal{K} \text{ and } g_{i,j} \in \mathcal{O}_p, \text{ for } i=1, \ldots, n}
\]
{(see remark \ref{rk:DarbouxFI})} globally defined on $M$ such that $ df_1 \wedge \dots \wedge d f_n \not\equiv 0$.
\end{definition}

A foliated $\mathcal{K}$-Darboux data $(M,{\omega},\mathcal{D},E)$ is said to be \textit{trivial} at a point $p$ if all functions $f_i$ in $\mathcal{D}$ are first integrals of ${\omega}$, i.e $\partial(f_i) \equiv 0 $ for all $\partial \in {\omega}_p$.\\
\\
Given an admissible blowing-up $\sigma: (\widetilde{M},\widetilde{{\omega}},\widetilde{E}) \to (M,{\omega},E)$, we denote by $\widetilde{D}$ the total transform of $\mathcal{D}$, i.e. $\widetilde{D} = \sigma^{\ast}\mathcal{D} = (f_1 \circ \sigma,\ldots, f_n\circ \sigma)$. In particular, this guarantees that $(\widetilde{M},\widetilde{{\omega}},\widetilde{D},\widetilde{E})$ is a foliated $\mathcal{K}$-Darboux data and we can write $\sigma : (\widetilde{M},\widetilde{{\omega}},\widetilde{\mathcal{D}},\widetilde{E}) \to (M,{\omega},\mathcal{D},E)$.

\subsection{Compact Notation}

In sections \ref{sec:mon} and \ref{sec:drop} it will be convenient to have a compact notation for denoting a collection of monomials. To this end, let $\pmb{u}$ be a collection of $k$ functions $(u_1,\dots,u_k)$ and $\pmb{A}$ be a $t \times k$ matrix:
\[
 \pmb{A}=\begin{bmatrix}
    \pmb{\alpha_1}\\
    \vdots \\
    \pmb{\alpha_t}
   \end{bmatrix}=
   \begin{bmatrix}
    \alpha_{1,1} & \dots & \alpha_{1,k}\\
    \vdots & \ddots &\vdots\\
    \alpha_{t,1} & \dots & \alpha_{t,k}
   \end{bmatrix}
\]
We define:
\[
 \umon{A}:=\begin{bmatrix}
    \umon{\pmb{\alpha_1}}\\
    \vdots \\
    \umon{\pmb{\alpha_t}}
   \end{bmatrix}=
   \begin{bmatrix}
    u_1^{\alpha_{1,1}} \cdots  u_k^{\alpha_{1,k}}\\
    \vdots \\
    u_1^{\alpha_{t,1}}  \cdots  u_k^{\alpha_{t,k}}
   \end{bmatrix}
\]
\begin{lemma}
Let $\pmb{u} = (u_1, \dots, u_k)$ and $\pmb{x} = (x_1, \dots x_r)$ be two collections of functions such that $\pmb{u} = \xmon{B}$ for some $k\times r$ matrix $\pmb{B}$. Then, for any $t \times k$ matrix $\pmb{A}$, we have that $\umon{A} = \xmon{ AB}$.
\label{lem:CompForm}
\end{lemma}
\begin{proof}
 Indeed, let $\pmb{\beta_i}$ be the line vectors of $\pmb{B}$, i.e
$
  \pmb{B}= \begin{bmatrix}
    \pmb{\beta_1}\\
    \vdots \\
    \pmb{\beta_k}
   \end{bmatrix}
$. Note that, by definition $u_i = \xmon{\beta_i}$, which implies that:
 \[
  \umon{A} =  \begin{bmatrix}
    u_1^{\alpha_{1,1}} \cdots  u_k^{\alpha_{1,k}}\\
    \vdots \\
    u_1^{\alpha_{t,1}}  \cdots  u_k^{\alpha_{t,k}}
   \end{bmatrix} = \begin{bmatrix}
    \pmb{x}^{\alpha_{1,1} \pmb{\beta_1}} \cdots 
    \pmb{x}^{\alpha_{1,k} \pmb{\beta_k}}\\
    \vdots \\
    \pmb{x}^{\alpha_{t,1} \pmb{\beta_1}} \cdots 
    \pmb{x}^{\alpha_{t,k} \pmb{\beta_k}}
   \end{bmatrix}=
   \begin{bmatrix}
   \Pi^r_{i=1} x_i^{\sum^{k}_{j=1} \alpha_{1,j}\beta_{j,i}}\\
    \vdots \\
    \Pi^r_{i=1} x_i^{\sum^k_{j=1} \alpha_{t,j}\beta_{j,i}}
   \end{bmatrix} =\xmon{C}
 \]
where
\[
 \pmb{C} =  \begin{bmatrix}
 \sum^{k}_{j=1} \alpha_{1,j}\beta_{j,1} & \dots &  \sum^{k}_{j=1} \alpha_{1,j}\beta_{j,r} \\
 \vdots & \ddots &\vdots\\
  \sum^{k}_{j=1} \alpha_{t,j}\beta_{j,1} & \dots &  \sum^{k}_{j=1} \alpha_{t,j}\beta_{j,r}
   \end{bmatrix}
\]
which is equal to $\pmb{A B}$.
\end{proof}

\section{Monomial singular distribution}
\label{sec:mon}

Let $\mathcal{K}$ be a {subfield of $\mathbb{K}$}.

\begin{definition}[Monomial singular distribution]
\label{def:MSD}
Given a foliated manifold $(M,{\omega},\allowbreak E)$, we say that the singular distribution ${\omega}$ is $\mathcal{K}$-\textit{monomial} at a point $p$ if there exists set of generators $\{\partial_1,...,\partial_d\}$ of ${\omega} \cdot \mathcal{O}_p$ and a coordinate system $(\pmb{u},\pmb{w}) = (u_1,\ldots,  u_r,w_{r+1},\allowbreak \ldots, w_m)$ centered at $p$ such that:

\begin{itemize}
\item[(i)] Locally $E = \{u_1 \cdots u_l=0\}$, for some $l\leq r$;
\item[({ii})] The vector fields $\partial_i$ are of the form:
\[
 \begin{aligned} 
  \partial_i &= \sum^{r}_{j=1} \alpha_{i,j}u_j \partial u_j \text{ } i=1, \ldots, s:=d+r-m \text{, and}\\
  \partial_i &=  \partial w_{r-s+i} \text{ }i = s+1, \ldots, d 
 \end{aligned}
\]
where $\alpha_{i,j} \in \mathcal{K}$ and $s$ is the number of vector fields in $\{\partial_1, \ldots,\partial_d\}$ which are singular.
\item[({iii})] If $\omega' \subset Der_M(-logE)$ is an involutive $d$-singular distribution such that ${\omega} \subset \omega'$, then ${\omega} = \omega'$.
\end{itemize}

In this case, we say that $(\pmb{u},\pmb{w})$ is a \textit{monomial coordinate system} and that $\{\partial_1,...,\partial_d\}$ is a \textit{monomial basis} of ${\omega}_p$.
\end{definition}
 
\begin{remark}[Geometrical Interpretation of $(iii)$]
Assuming conditions $[i-ii]$ above, Property $[iii]$ implies that the singularity set of ${\omega}$ is of codimension at least two outside of the exceptional divisor $E$.
\label{rk:Condiv}
\end{remark}
 
\begin{notation}[Monomial coordinate system]
We sometimes need to distinguish one of the non-exceptional coordinates $\pmb{w}$. To this end, we denote by $(\pmb{u},v,\pmb{w}) = {\allowbreak(u_1,\ldots,  u_r,v,w_{r+2},\allowbreak \ldots, w_m)}$ a monomial coordinate system where the vector field $\partial_{v}$ is always assumed to be contained in ${\omega}_p$.
\end{notation}
 
The importance of this class of singular distributions for our propose is enlightened by the following result:
 
\begin{lemma}[Monomial first integrals]
\label{lem:Fi}
Given a foliated manifold $(M,{\omega},E)$, the singular distribution ${\omega}$ is $\mathcal{K}$-monomial if and only if for any monomial coordinate system $(\pmb{u},\pmb{w}) = (u_1,\ldots,  u_r,w_{r+1}, \ldots w_m)$ centered at $p$, there exists $m-d$ (complex {multi-}valued) monomials $\umon{B} = (\umon{\beta_1},\dots,\umon{\beta_{m-d}})$, where the matrix $\pmb{B}$ has maximal rank and entries in $\mathcal{K}$, such that
 \[
{\omega}_p =\{\partial \in Der_p(-log\, E);\text{ } \partial(\umon{\beta_i})\equiv 0 \text{ for all } i\}
\]
In this case, we call $\umon{B}$ a complete system of first integrals.
\end{lemma}
\begin{proof}
First, let us assume that ${\omega}$ is a $\mathcal{K}$-monomial singular distribution and let us fix a point $p$ in $M$ and a monomial coordinate system $(\pmb{u},\pmb{w})$. Note that if $f$ is a first integral of ${\omega}$, then it can not depend on any coordinate $w$, since all the derivations $\partial_{w_i}$ are contained in the stalk ${\omega}_p$. So, consider a monomial $\umon{\beta}$ and let us remark that:
\[
 \begin{aligned}
  \partial_i(\umon{\beta}) &\equiv \umon{\beta} \sum^{r}_{j=1} \alpha_{i,j}\beta_i \text{ for }i=1,\ldots, s \text{, and}\\
   \partial_i(\umon{\beta}) &\equiv 0\text{, otherwise}
 \end{aligned}
\]
So, the monomial $\umon{\beta}$ is a first integral of ${\omega}$ if, and only if:
\begin{equation}
 \sum^{r}_{j=1} \alpha_{i,j}\beta_i = 0 \text{, for }i = 1, \dots, s
 \label{eq:L1}
\end{equation}
Thus, there exists a $r- s = r-(d+r-m) =  m-d$ linear subspace $L$ of $\mathcal{K}^{r}$ that contains all vector $\pmb{\beta}$ satisfying the equations $\ref{eq:L1}$. In particular, we can choose a system of generators $\{\pmb{\beta_1}, \dots, \pmb{\beta_{m-d}}\}$ of $L$. So, the $m-d$ monomials $\umon{B} = (\umon{\beta_1},\dots,\umon{\beta_{m-d}})$ are first integrals of ${\omega}$ and:
\[
 {\omega}_p \subset \{\partial \in Der_p(-logE);\text{ } \partial(\umon{\beta_i})\equiv 0 \text{ for all } i\leq m-d\}
\]
By the maximal condition (iv), we conclude that both singular distributions are equal. Now, let ${\omega}$ be a singular distribution whose stalk at $p$ is given by 
\[                                                             
\{\partial \in Der_p(-logE); \text{ }\partial(\umon{\beta_i})\equiv 0 \text{ for all } i\leq m-d\}
\]
and let us prove that ${\omega}$ is a $\mathcal{K}$-monomial singular distribution. First, the vector fields $\partial_i =  \partial w_{r-s+i}$ for $i = s+1, \ldots, d $ are all contained in ${\omega}$. So, consider a vector-field of the form  ${\partial} = \sum_{j=1}^r \alpha_j u_j \partial_{u_j}$ and let us note that:
\[
 \partial(\umon{\beta_i}) = \umon{\beta_i}\left(\sum_{j=1}^k \alpha_j \beta_{i,j} \right)
\]
Since $\partial$ is tangent to $E$, it belongs to ${\omega}$ if and only if:
\begin{equation}
 \sum_{j=1}^r \alpha_j \beta_{i,j} =0 \text{ for }i = 1, \ldots, m-d
 \label{eq:L2}
\end{equation}
Consider the $r- (m-d) = d+r-m =s$ linear subspace $L$ of $\mathcal{K}^{r}$ that contains all vectors $\pmb{\alpha}$ satisfying the equations $\ref{eq:L2}$. In particular, fix system of generators $\{\pmb{\alpha}_{1}, \dots, \pmb{\alpha}_{s}\}$ of $L$ and let $ \partial_i = \sum_{j=1}^r \alpha_{i,j}u_j \partial_{u_j}$, which are vector fields contained in ${\omega}_p$. We now just need to prove that $\{\partial_1,\ldots, \partial_d\}$ generates ${\omega}_p$. Indeed, let $\partial= \sum_{j=1}^{r} A_j \partial_{u_j} + \sum_{j=r+1}^m B_j \partial_{w_j}$ be an analytic vector-field contained in ${\omega}_p$. Then
\[
 \partial(\umon{\beta_i})\equiv 0 \implies \umon{\beta_i} \sum_{j=1}^r \beta_{i,j} \frac{A_j}{u_j} \equiv 0
\]
which implies that the formal vector field $\widehat{\partial}$ is contained in the formal distribution generated by $\{\widehat{\partial_1}, \ldots, \widehat{\partial_d}\}$, i.e
\[
\widehat{\partial} = \sum_{i=1}^{s} \widehat{C_i} \widehat{\partial_i} + \sum_{i=s+1}^d \widehat{B_i} \widehat{\partial_{i}}
\]
for some power series $\widehat{C}_i$. Now, let $\pmb{\gamma_i}$ be a multiindex in $\mathcal{K}^r$ such that $\partial_{i}(\umon{\gamma_i}) =\umon{\gamma_i}$ and $\partial_{j}(\umon{\gamma_i}) =0$ if $j\neq i$. Then:
\[
\widehat{\partial}(\umon{\gamma_i}) = \widehat{C_i} \umon{\gamma_i}
\]
which implies that $\partial(\umon{\gamma_i}) = \umon{\gamma_i} C_i$, where $C_i$ has the same formal expression of $\widehat{C_i}$ and is analytic. We conclude that $\partial = \sum_{i=1}^{s} C_i \partial_i + \sum_{i=s+1}^d B_i \partial_{i} $, which finishes the proof.
\end{proof}
 
We now turn to two more important results in of monomial singular distributions:
 
\begin{lemma}[Openness of monomiality]
The $\mathcal{K}$-monomiality is an open condition i.e. if ${\omega}$ is $\mathcal{K}$-monomial at $p$ in $M$, then there exists an open neighborhood $U$ of $p$ such that ${\omega}$ is $\mathcal{K}$-monomial at every point $q$ in $U$. Moreover, if $(\pmb{u},\pmb{w})$ is a coordinate system defined in a connected open neighborhood $V$ which is monomial at $p$, then ${\omega}$ is $\mathcal{K}$-monomial everywhere in $V$.
\label{lem:Gmonloc}
\end{lemma}
\begin{proof}
For a proof coming directly from the definition, see \cite{BeloT} Lemma 2.2.1 or \cite{Bel2} Lemma 3.6. In here, we present a different proof (using Lemma \ref{lem:Fi}) whose reasoning is useful for the current paper (see Lemma \ref{lem:claim1}).\\
\\
Fix a monomial coordinate system $(\pmb{u},\pmb{w}) = (u_1,\ldots,  u_r,w_{r+1}, \ldots,\allowbreak w_m)$ centered at $p$ which is defined in a connected neighborhood $U$ of $p$. Since ${\omega}$ is $\mathcal{K}$-monomial, by Lemma \ref{lem:Fi}, there exists $m-d$ monomials $\umon{B} = (\umon{\beta_1},\dots,\umon{\beta_{m-d}})$, such that
 \[
{\omega}_p =\{\partial \in Der_p(-logE);\text{ } \partial(\umon{\beta_i})\equiv 0 \text{ for } i\leq m-d\}
\]
where $\pmb{\beta_i} \in \mathcal{K}^r$. So, fix a point $q \in U$ and let $(\pmb{\xi},\pmb{\zeta})$ be its coordinate in the coordinate system $(\pmb{u},\pmb{w})$. Apart from re-indexing, we can assume that $\pmb{\xi} = (0, \dots, 0, \xi_{t+1},\dots, \xi_k)$ for some $t\leq k$. Furthermore, in the real case we can assume that $\xi_i<0$ by considering the changes $x_i=-x_i$ (and the respective change in first integral). Consider the coordinate system $(\pmb{x},\pmb{y},\pmb{v}) = (x_1, \dots, x_t,\allowbreak y_{t+1}, \dots, y_r, v_{r+1},\dots v_{m})$ where
\[
\begin{aligned}
 x_i &=u_i\\
 y_i&=u_i-\xi_i\\
 v_i&=w_i-\zeta_i
\end{aligned}
\]
which is a coordinate system centered at $q$. We can now write:
\[
 \umon{B} = \xmon{B_1}(\pmb{y}-\pmb{\xi})^{\pmb{B_2}}
\]
where $\pmb{B_1}$ is a $r \times t$ matrix and $\pmb{B_2}$ is a $r \times (m-d-t)$ matrix such that
\[
 \pmb{B} = \begin{bmatrix}
                   \pmb{B_1} & \pmb{B_2}
                  \end{bmatrix}
\]
Furthermore, apart from re-ordering the lines of the matrix $\pmb{B}$, we can further write:
\[
 \pmb{B} = \begin{bmatrix}
 \pmb{B^{'}_1}&\pmb{B^{'}_2} \\
 \pmb{B^{''}_1}&\pmb{B^{''}_2}
\end{bmatrix}
\]
where $ \pmb{B_1} = \begin{bmatrix} \pmb{B^{'}_1} \\ \pmb{B^{''}_1}\end{bmatrix}$ and the rank of $\pmb{B^{'}_1}$ is maximal and equal to the rank of $\pmb{B_1}$. So, there exists a change of coordinates $(\pmb{x}\sm{1},\pmb{y}\sm{1},\pmb{v}\sm{1})$ such that:
\[
 \umon{B} = \pmb{x}\sm{1}^{\pmb{C_1}}(\pmb{y}\sm{1}-\pmb{\xi})^{\pmb{C_2}}
\]
where:
\[
 \pmb{C} = \begin{bmatrix}
                   \pmb{C_1} & \pmb{C_2}
                  \end{bmatrix} = \begin{bmatrix}
 \pmb{C^{'}_1}&\pmb{C^{'}_2} \\
 \pmb{C^{''}_1}&\pmb{C^{''}_2}
\end{bmatrix} = \begin{bmatrix}
 \pmb{B^{'}_1}& 0  \\
 \pmb{B^{''}_1}&\pmb{\Lambda}
\end{bmatrix}
\]
where $\Lambda$ is a maximal rank matrix with entries in $\mathcal{K}$. This implies that the collection $(\pmb{x}\sm{1}^{\pmb{B_1^{'}}},\pmb{x}\sm{1}^{\pmb{B_1^{''}}}(\pmb{y}\sm{1}-\pmb{\xi})^{\pmb{\Lambda}} )$ is a collection of first integrals of ${\omega}_q$. Since $\pmb{B_1^{'}}$ has rank equal to $\pmb{B_1}$, we conclude that:
\[
(\pmb{x}\sm{1}^{\pmb{B_1^{'}}},(\pmb{y}\sm{1}-\pmb{\xi})^{\pmb{\Lambda}})
\]
is another collection of first integrals of ${\omega}_q$. Furthermore, since $\pmb{\Lambda}$ is of maximal rank, there exists a coordinate system $(\pmb{x}\sm{2},\pmb{y}\sm{2},\pmb{z}\sm{2},\pmb{v}\sm{2})$ where $\pmb{x}\sm{2} = \pmb{x}\sm{1}$ and $\pmb{v}\sm{2}=\pmb{v}\sm{1}$ such that:
\[
 (\pmb{y}\sm{1}-\pmb{\xi})^{\pmb{\Lambda}} = \pmb{y}\sm{2}-\pmb{\xi}\sm{2}
\]
which finally implies that the monomial functions 
\[
(\pmb{x}\sm{1}^{\pmb{B_1^{'}}},\pmb{y}\sm{2})
\]
are first integrals of ${\omega}_q$. By the analyticity of ${\omega}$ and Lemma \ref{lem:Fi}, we conclude that the singular distribution ${\omega}_q$ is monomial. Since $q$ is an arbitrary point in $U$, we conclude that the monomiality property is open.
\end{proof}

\begin{lemma}
Let $(M,{\omega},E)$ be a foliated manifold where ${\omega}$ is $\mathbb{Q}$-monomial at point $p$. Then, there exists an open neighborhood $U$ of $p$ and a sequence of admissible blowings-up $\tau :(\widetilde{U},\widetilde{{\omega}},\widetilde{E}) \to (M,{\omega},E)$ such that, at every point $q$ in the pre-image of $p$, there exists a monomial coordinate system $(\pmb{u},\pmb{w}) = (u_1,\ldots,  u_r,w_{r+1}, \ldots w_m)$ centered at $q$ and $m-d$ analytic monomials $\umon{B} = (\umon{\beta_1},\dots,\umon{\beta_{m-d}})$, where the matrix $\pmb{B}$ has maximal rank and entries in $\mathbb{N}$, such that
 \[
{\omega}_p =\{\partial \in Der_p(-logE);\text{ } \partial(\umon{\beta_i})\equiv 0 \text{ for all } i\}
\]
In this case, we call $\umon{B}$ a complete system of analytic first integrals.
\label{lem:Final}
\end{lemma}
\begin{proof}
Fix a monomial coordinate system $(\pmb{u},\pmb{w}) = (u_1,\ldots,  u_r,w_{r+1}, \ldots,\allowbreak w_m)$ centered at $p$ which is defined in a connected neighborhood $U$ of $p$. Since ${\omega}$ is $\mathbb{Q}$-monomial, by Lemma \ref{lem:Fi}, there exists $m-d$ monomials $\umon{B} = (\umon{\beta_1},\dots,\umon{\beta_{m-d}})$, such that
 \[
{\omega}\cdot \mathcal{O}_p =\{\partial \in Der_p(-logE);\text{ } \partial(\umon{\beta_i})\equiv 0 \text{ for } i\leq m-d\}
\]
and apart from taking a multiple of the multiindexes $\pmb{\beta_i}$, we can assume that $\pmb{\beta_i} \in \mathbb{Z}^r$. So, we conclude that there exists two multiindexes $\pmb{\delta_i}$ and $\pmb{\gamma_i}$ in $\mathbb{N}^r$ such that $\pmb{\beta_i} = \pmb{\delta_i} - \pmb{\gamma_i}$. Consider the ideal $I$ generated by:
\[
I = (\umon{\delta_i}\umon{\gamma_i}(\umon{\delta_i}-\umon{\gamma_i}))
\]
and let $\tau :(\widetilde{U},\widetilde{{\omega}},\widetilde{E}) \to (M,{\omega},E)$ be a principalization of $I$, where $U$ is a sufficiently small neighborhood of $p$ where $I$ is well-defined. The sequence of blowings-up $\tau$ may be chosen to be \emph{combinatorial} in respect to the exceptional divisor $F:= \{\Pi_{i=1}^r u_i\}$, i.e. $\tau$ is a composition of blowings up with centers that are strata of the divisor $F$ and its total transforms. So, we can cover the $\widetilde{U}$ by affine charts with coordinate system $(\pmb{x},\pmb{w})$ centered at a point $q$ such that:
\[
\pmb{u} = \pmb{x}^{\pmb{A}} \text{ where }\pmb{A}=
\begin{bmatrix}
 a_{1,1} & \ldots  & a_{1,r} \\
 \vdots &  \ddots & \vdots \\
a_{r,1} & \ldots  & a_{r,r} 
\end{bmatrix}
\]
By Lemma \ref{lem:CompForm}, $\pmb{u}^{\pmb{B}} = \pmb{x}^{\pmb{AB}}$ and ${\omega}$ is $\mathbb{Q}$-monomial at $q$. Furthermore, since $\tau^{\ast}(I)$ is principal, we conclude that either $\xmon{A\beta_i}$ or $\xmon{-A\beta_i}$ is analytic, which implies we can choose analytic monomial first integrals of $\widetilde{{\omega}} \cdot \mathcal{O}_q$ (i.e. without poles). Now, by Lemma \ref{lem:Gmonloc} and analiticity of the first integrals, we conclude that we can choose analytic monomial first integrals of $\widetilde{{\omega}}$ at any point in the pre-image of $p$.
\end{proof}

\section{${\omega}$-admissible Blowings-up}
\label{sec:tadm}
 
Given an ideal sheaf $\mathcal{I}$, we consider ideal sheaves $\Gamma_{{\omega},k}(\mathcal{I})$, which we call \textit{generalized $k$-Fitting ideal}, whose stalk at each point $p$ in $M$ is generated by all terms of the form:
\[
det \left\| 
\begin{array}{ccc}
\partial _1(f_1) & ... & \partial_{1}(f_k) \\
 \vdots & \ddots & \vdots \\
 \partial_{k}(f_1) & ... & \partial_{k}(f_k) 
\end{array}
\right\|
\]
for $\partial_i \in {\omega} \cdot \mathcal{O}_{p}$ and $f_j \in \mathcal{I}\cdot \mathcal{O}_{p}$.
 
\begin{definition}[${\omega}$-admissible blowing-up] We say that an admissible blowing-up $\sigma:(\widetilde{M},\widetilde{{\omega}},\widetilde{E}) \to (M,{\omega},E)$ is ${\omega}$\textit{-admissible} if there exists $d_0 \in \mathbb{N}$ such that:
\begin{enumerate}
\item The generalized $k$-Fitting ideal $\Gamma_{{\omega},k}(\mathcal{I}_{\mathcal{C}})$ is equal to $\mathcal{O}_M$ for $k\leq d_0$;
\item The ideal $\Gamma_{{\omega},k}(\mathcal{I}_{\mathcal{C}}) + \mathcal{I}_{\mathcal{C}}$ is {equal to} $\mathcal{I}_{\mathcal{C}}$ for $k>d_0$. 
\end{enumerate} 
where $\mathcal{I}_{\mathcal{C}}$ is the reduced ideal sheaf whose support is the center of the blowing-up. We say that the blowing up is  ${\omega}$\textit{-invariant}, moreover, if $d_0 =0$.  
\label{def:tadm}
\end{definition}

The following result enlightens the interest of ${\omega}$-admissible blowings-up:

\begin{proposition}
\cite[Theorem 4.1.1]{BeloT} or \cite[Proposition 4.4]{Bel2}. Let $(M,{\omega},E)$ be a $\mathcal{K}$-monomial foliated manifold and $
\sigma: (\widetilde{M},\widetilde{{\omega}},\widetilde{E})\allowbreak \to (M,{\omega},E)
$ be a ${\omega}$-admissible blowing-up. Then $\widetilde{{\omega}}$ is also $\mathcal{K}$-monomial.
\label{prop:AdmCenter}
\end{proposition}
Before continuing, let us present a couple of examples in order to illustrate the definition:
\begin{example}
We present four examples:
\begin{enumerate}
\item If the center $\mathcal{C}$ is ${\omega}$-invariant center (i.e if all leaves of ${\omega}$ that intersects $\mathcal{C}$ are contained in $\mathcal{C}$), the blowing-up is ${\omega}$-admissible.\\
\item If the center $\mathcal{C}$ is an admissible ${\omega}$-totally transverse (i.e all vector fields in ${\omega}$ are transverse to $\mathcal{C}$), the blowing-up is ${\omega}$-admissible.\\
\item Let $M = \mathbb{C}^3$ and ${\omega}$ be generated by $ \{\partial_x, \partial_y\}$. A blowing-up with center $\mathcal{C} = \{x=0, z=0\}$ is ${\omega}$-admissible since $\Gamma_{{\omega},1}(\mathcal{I}_{\mathcal{C}}) = \mathcal{O}_M$ and $\Gamma_{{\omega},2}(\mathcal{I}_{\mathcal{C}}) \subset \mathcal{I}_{\mathcal{C}}$.\\
\item Let $M = \mathbb{C}^3$ and ${\omega}$ be generated by $ \{\partial_x, \partial_y\}$. A blowing-up with center $\mathcal{C} = \{x^2-z=0,y=0\}$ is not ${\omega}$-admissible since $\Gamma_{{\omega},2}(\mathcal{I}_{\mathcal{C}}) = (x,y,z)$.
\end{enumerate}
\end{example}
 
\subsection{Foliated ideal sheaves and ${\omega}$-admissible resolution of singularities}
 
A \textit{foliated ideal sheaf} is a quadruple $(M,{\omega},\mathcal{I},E)$ where $\mathcal{I}$ is a coherent and everywhere non-zero ideal sheaf of $\mathcal{O}_M$. Given an admissible blowing-up $\tau: (\widetilde{M},\widetilde{{\omega}},\widetilde{E}) \to (M,{\omega},E)$, we define the transform $\widetilde{\mathcal{I}}$ of $\mathcal{I}$ as the total transform $\mathcal{I}\cdot\mathcal{O}_{\widetilde{M}}$.\\
\\
We present two results which can be found in \cite{Bel2} based on the notion of ${\omega}$-admissible blowings-up. Both results are important technical steps for this work:

\begin{theorem}[${\omega}$-Invariant resolution of Ideal] \cite[Theorem 4.1.1]{BeloT} or \cite[Lemma 7.1]{Bel2}. Let $(M,{\omega}, \mathcal{I},E)$ be a foliated ideal sheaf and $M_0$ a relatively compact open set of $M$. Suppose that $\mathcal{I}_0:=\mathcal{I}\cdot\mathcal{O}_{M_0}$ is invariant by ${\omega}_0:={\omega}\cdot \mathcal{O}_{M_0}$, i.e., ${\omega}_0 [ \mathcal{I}_0]\allowbreak \subset \mathcal{I}_0$. Then, there exists a sequence of ${\omega}$-admissible blowings-up:
\[
\begin{tikzpicture}
  \matrix (m) [matrix of math nodes,row sep=3em,column sep=3em,minimum width=1em]
  {(\widetilde{M},\widetilde{{\omega}},\widetilde{\mathcal{I}},\widetilde{E}) = (M_r,{\omega}_r,\mathcal{I}_r,E_r) & \cdots & (M_0,{\omega}_0,\mathcal{I}_0,E_0)\\};
  \path[-stealth]
    (m-1-1) edge node [above] {$\sigma_r$} (m-1-2)
    (m-1-2) edge node [above] {$\sigma_1$} (m-1-3);
\end{tikzpicture}
\]
such that $\widetilde{\mathcal{I}}$ is principal with support contained in $\widetilde{E}$. In particular, if ${\omega}$ is $\mathcal{K}$-monomial, then $\widetilde{{\omega}}$ is $\mathcal{K}$-monomial.
\label{thm:RinvI}
\end{theorem}

\begin{theorem}[${\omega}$-Resolution of Ideal] \cite[Theorem 1.3]{Bel2}. Let $(M,{\omega},\mathcal{I},E)$ be a foliated ideal sheaf. Then, for every point $p$ in $M$, there exists a ${\omega}$-admissible {complete} collection of local blowings-up (i.e all local blowings-up are ${\omega}$-admissible - see subsection \ref{ssec:CLB})
\[
\tau_i : (M_i,{\omega}_i,\mathcal{I}_i,E_i) \rightarrow (M,{\omega},\mathcal{I},E)
\]
such that the ideal sheaf $\mathcal{I}_i$ is a principal ideal sheaf with support contained in $E_i$. In particular, if ${\omega}$ is $\mathcal{K}$-monomial, then ${\omega}_i$ is $\mathcal{K}$-monomial.
\label{thm:RI}
\end{theorem}

\section{Main Invariant of a foliated Darboux data}
\label{sec:MainInv} 

A foliated $\mathcal{K}$-Darboux data $(M,{\omega},\mathcal{D},E)$ will be called a $\mathcal{K}$-monomial foliated Darboux data if ${\omega}$ is $\mathcal{K}$-monomial. We start by simplifying the expressions of the first integrals via resolution of singularities: 

\begin{lemma}
Let $(M,{\omega},\mathcal{D},E)$ be a $\mathcal{K}$-monomial foliated Darboux data and consider monomial coordinate systems $(\pmb{u},\pmb{w})$ of $p$. Then, {by} a {complete} collection of local ${\omega}$-admissible blowings-up, we can {reduce to the case that} there exists {multi-indexes} $\pmb{\delta}_i$ with entries in $\mathcal{K}$ such that:
\begin{equation}
 f_i = g_i + \umon{\delta_i}T_i
 \label{eq:Basica}
\end{equation}
where $g_i$ are first integrals of ${\omega}$; the function $T_i$ are analytic; and all monomials in the Taylor expansion of $\umon{\delta_i}T_i$ are \textit{not} first integrals of ${\omega}$.
\label{lem:eqBasica}
\end{lemma}
\begin{proof}
This result will follow from simultaneous ${\omega}$-admissible resolution of singularities. Indeed, by assumption $
f_i = \prod_j g_{i,j}^{k_{i,j}}
$, where $g_{i,j}$ are analytic functions and $k_{i,j} \in \mathcal{K}$. Consider the ideal:
\[
\mathcal{I} = \left( \prod_{i,j} g_{i,j}\right)
\]
which is locally defined in some open set $U$ of $p$. By theorem \ref{thm:RI}, there exists a ${\omega}$-admissible {complete} collection of local blowings-up $\tau_l : (U_l,{\omega}_l,\mathcal{I}_l,E_l) \rightarrow (U,{\omega},\mathcal{I},E)$ which principalize $\mathcal{I}$. In particular:
\[
\tau^{\ast}_{l}(g_{i,j}) = \umon{\gamma_{i,j,l}} U_{i,j,l}
\]
where $\pmb{\gamma}_{i,j,l}$ is a multi-index with entries in $\mathbb{Q}$ and $U_{i,j,l}$ is an analytic unit. Therefore:
\[
\tau^{\ast}_{l}(f_{i}) = \umon{ \delta_{i,l}} U_{i,l}
\]
where $\pmb{\delta}_{i,l} = \sum_j k_{i,j} \gamma_{i,j,l}$ and $U_{i,l} = \prod U_{i,j,l}^{k_{i,j}}$ is an analytic function. Since $U_{i,l}$ is analytic, we can use its Taylor expansion in order to get:
\[
\tau^{\ast}_{l}(f_{i}) = g_{i,l} + \umon{\delta_{i,l}}T_{i,l}
\]
where $g_{i,l}$ are first integrals of ${\omega}_l$ and all monomials in the Taylor expansion of $\umon{\delta_i}T_{i,l}$ are \textit{not} first integrals of ${\omega}$.
\end{proof}

\begin{definition}[Main Invariant]\label{def:invariant}
Let $(M,{\omega},\mathcal{D},E)$ be a $\mathcal{K}$-monomial foliated Darboux data. We will say that $\nu(p, {\omega},\mathcal{D}) = \infty$ if equation \ref{eq:Basica} is not satisfied. Otherwise, we consider the coordinate dependent function
\[
\nu(p, {\omega},\mathcal{D},(\pmb{u},\pmb{w})):=\min\{|\pmb{\lambda}|:\ \partial^{\pmb{\lambda}}_{\pmb{w}} T_i\text{ is a unit}\}.
\]
where we assume that $\partial_{\pmb{w}}^{\pmb{\lambda}}$ is the identity if $\pmb{w}$ is empty, and $\nu(p)=\infty$ if there are no $\pmb{\lambda}$ such that $\partial^{\pmb{\lambda}}_{\pmb{w}} T_i$ is a unit. We define the \textit{tangency order} of $(M,{\omega},\mathcal{D},E)$ by:
\[
\nu(p, {\omega},\mathcal{D}):=\min\{\nu(p, {\omega},\mathcal{D},(\pmb{u},\pmb{w})):\ \text{for all }(\pmb{u},\pmb{w}) \}.
\]
When there is no risk of confusion, we simply denote $\nu(p, {\omega},\mathcal{D})$ by $\nu(p)$.
\end{definition}
 
In what follows, we consider $\nu$ as the main invariant. Note that this invariant is upper {semi-continuous in $M$} (since the $T_i$'s are analytic). We now present a normal form which will be used in the remainder of the paper, in part so to fix notation that will be used in the following sections.
 
\begin{lemma}[Weierstrass-Tschirnhausen normal form]\label{def:basicnormalform}
Let $p$ be a point of $M$ where the invariant $\nu = \nu(p,{\omega},\mathcal{D})$ is finite and bigger than one, i.e $1< \nu < \infty$. Then, there exists a monomial coordinate system $(\pmb{u},v,\pmb{w}){=\allowbreak(u_1,\ldots,  u_r,v,w_{r+2},\allowbreak \ldots, w_m)}$ at $p$ such that the functions $T_i$ are given by:
\begin{equation}
 \label{eq:basicnormaform}
\begin{aligned}
T_1 &= v^{\nu} U  + \sum^{\nu-2}_{j=0} a_{1,j}(\pmb{u},\pmb{w}) v^j\text{ where $U$ is an unit, and}\\
T_i &= v^{\nu} \bar{T}_i  + \sum^{\nu-1}_{j=0} a_{i,j}(\pmb{u},\pmb{w}) v^j
\end{aligned}
\end{equation}
and the vector-field $\partial_{v}$ belongs to ${\omega}_p$.
\label{lem:BasicNormalForm}
\end{lemma}
\begin{proof}
Since the invariant is finite, there exists a coordinate system $(\pmb{u},v,\pmb{w})$ of $p$ such that the vector-field $\partial_{v}$ belongs to ${\omega}_p$ and, apart from re-indexing, the function $\partial^{\nu}_{v^{\nu}} T_1$ is a unit. Furthermore, by the implicit function Theorem, there is a change of coordinates $(\widetilde{\pmb{u}},\widetilde{v},\widetilde{\pmb{w}}) = (\pmb{u},V(\pmb{u},v,\pmb{w}),\pmb{w})$ such that $\partial^{\nu-1}_{\widetilde{v}^{\nu-1}} T_1(\widetilde{\pmb{u}},0,\widetilde{\pmb{w}}) \equiv 0$. Thus:
 
\[
  \begin{aligned}
  T_1 &= \widetilde{v}^{\nu} U + \sum^{\nu-2}_{j=0} \widetilde{v}^j a_{1,j}(\widetilde{\pmb{u}},\widetilde{\pmb{w}}) \text{ where $U$ is an unit, and}\\
  T_i &= \widetilde{v}^{\nu} \bar{T}_i + \sum^{\nu-1}_{j=0} \widetilde{v}^j a_{i,j}(\widetilde{\pmb{u}},\widetilde{\pmb{w}})
 \end{aligned}
\]
Finally, since $\widetilde{\pmb{u}} = \pmb{u}$ and $\widetilde{\pmb{w}} = \pmb{w}$, we have that $ \partial_{v} =  U \partial_{\widetilde{v}}$ for some unit $U$. This implies that $\partial_{\widetilde{v}}$ is contained in ${\omega}_p$, which proves the Lemma.
\end{proof}

\section{Proof of Theorem \ref{thm:main}}
\label{sec:Overview1}

The main result of this work follows from the following (technical) theorem:
 
\begin{theorem}\label{thm:main2}
Let $(M,{\omega},\mathcal{D},E)$ be a non-trivial $\mathcal{K}$-monomial Darboux data. Then, for each point $p$ in $M$, there exists a ${\omega}$-admissible {complete} collection of local blowings-up $\tau_i : (M_i,{\omega}_i,\mathcal{D}_i,E_i) \rightarrow (M,{\omega},\mathcal{D},E)$ such that, for every point $q_i$ in the pre-image of $p$, the invariant $\nu(q_i,{\omega}_i,\mathcal{D}_i)$ is zero or one.
\end{theorem}

The proof of this result follows from three main steps. We will present these steps in section \ref{sec:Overview} and we prove them in sections \ref{sec:Infinte}, \ref{sec:prep} and \ref{sec:drop}. In the rest of this section we use this result to prove Theorem \ref{thm:main}. We start by showing why the result is useful:

\begin{lemma}[Invariant zero or one]
\label{lem:invariantzerodone}
Let $(M,{\omega},\mathcal{D},E)$ be a $\mathcal{K}$-monomial foliated Darboux data and $p$ a point of $M$ where the invariant $\nu(p,{\omega},\mathcal{D})$ is $0$ or $1$. Then, there exists an index $i_0$ and a monomial coordinate system $(\pmb{u},\pmb{w}){=(u_1,\ldots,  u_r,}\allowbreak {w_{r+1},\ldots, w_m)}$ such that:
\[
 f_{i_{0}} = g_{i_0} + \umon{\delta_{{i_0}}}w_{{m}}^{\epsilon}
\]
where we recall that $g_{i_0}$ is a first integral of ${\omega}$ and $\umon{\delta_{{i_0}}}w_{{m}}^{\epsilon}$ is not a first integral (see equation \ref{eq:Basica}). Moreover the constant {$\epsilon = \nu \in \{0,1\}$}. In particular, the foliated Darboux-data $(M,\omega,\mathcal{D},E)$ given by $\omega_p = \{\partial \in {\omega}_p ; \partial(f_{i_0}) \equiv 0\}$ is $\mathcal{K}$-monomial.
\end{lemma}
\begin{proof}
Fix a monomial system of coordinates $(\pmb{u},\pmb{w})$ and recall that,  by Lemma \ref{lem:Fi}, there exists a complete system of first integrals $\umon{B} = (\umon{\beta_1}, \dots, \umon{\beta_{m-d}})$ of ${\omega}$. We now consider the cases where $\nu(p)$ is zero and one separately:\\
\\
First, assume that $\nu(p)=0$ (this is the case when $\epsilon =0$). Without loss of generality, we can assume that $T_1$ is a unit. By the definition of the functions $T_i$, the multi-index $\pmb{\delta}_1$ has to be linearly independent with all the multi-indexes $\pmb{\beta_i}$. Thus, apart from a change of coordinates (which preserves all monomials), we can assume that $T_1 = 1$ . So, the singular distribution $\omega= \{\partial \in {\omega} ; \partial(f_{1}) \equiv 0\}$ has a complete system of first integrals given by $(\umon{B},\umon{\delta_1})$ which implies that it is $\mathcal{K}$-monomial.\\
\\
Now, assume that $\nu(p)=1$ (this is the case when $\epsilon =1$). In this case, there exists a coordinate system $(\pmb{u},v,\pmb{w})$ such that $\partial_v T_1$ is a unit. So, apart from a change of coordinates in the $v$ coordinate, we can assume that $T_1 = v$. Thus, the singular distribution $\omega= \{\partial \in {\omega} ; \partial(f_{1}) \equiv 0\}$ has a complete system of first integrals given by $(\umon{B},\umon{\delta_1}v)$ which implies that it is $\mathcal{K}$-monomial.
\end{proof}

\subsection{Proof of Theorem \ref{thm:main} (Assuming Theorem \ref{thm:main2})}

Fixed the point $p$, recall that there exists $n$ first integrals $(f_1, \dots, f_n)$ of $\theta$ such that
\[
 d f_1 \wedge \dots \wedge df_n \neq 0
\]
So, let us consider a $\mathcal{K}$-monomial $m$-foliated Darboux data $(M,\omega\sm{0},\mathcal{D},E)$ where $\omega\sm{0}$ is the monomial singular distribution $Der_M(-logE)$ i.e the sheaf of derivations of $M$ tangent to $E$.

In this case, let us note that the singular distribution $\theta \cap Der_M(-logE)$ is obviously contained in $\omega\sm{0}$. The proof follows a recursive argument:

\begin{claim}
Let $\theta$ be a singular distribution with $n$ first integrals $(f_1, \dots, f_n)$ and $(M,\omega\sm{k},\mathcal{D},E)$ be a $\mathcal{K}$-monomial $(m-k)$-foliated Darboux data with $k<n$ such that:
\begin{enumerate}
 \item $\mathcal{D}$ is given by the $n$ first integrals $(f_1,\dots, f_n)$ of $\theta$;
 \item Apart from re-indexing the functions $(f_1,\dots, f_n)$, the singular distribution $\omega\sm{k}$ is equal to $\{\partial \in Der_M(-logE); \partial(f_i)\equiv 0 \text{ for all }i\leq k\}$. In particular $\theta \cap Der_M(-logE) \subset \omega\sm{k}$. 
\end{enumerate}
Then, for every point $q$ in $M$, there exists a collection of $\omega\sm{k}$-admissible local blowings-up:
\[
 \Phi_i : (M_i,\omega_i\sm{k},\mathcal{D}_i,E_i) \rightarrow (M,\omega\sm{k},\mathcal{D}, E) 
\]
such that, for each point $q_i$ in the pre-image of $q$, there exists a $\mathcal{K}$-monomial $[m-(k+1)]$-foliated Darboux data $(M_i,\omega_i\sm{k+1},\mathcal{D}_i,E_i)$ that satisfies properties $(1)$ and $(2)$ in respect to the strict transform $\theta_i$, i.e:
\begin{enumerate}
 \item $\mathcal{D}_i$ is given by the $n$ first integrals $\tau_i^{\ast}(f_1,\dots, f_n) = (f_1^\ast,\dots,\allowbreak f_n^\ast)$ of $\theta_i$;
 \item Apart from re-indexing the functions $(f_1^\ast,\dots,f_n^\ast)$, the singular distribution $\omega\sm{k+1}$ is equal to $\{\partial \in Der_{M_i}(-logE_i); \partial(f^{\ast}_i)\equiv 0 \text{ for all }i\leq k+1\}$. In particular $\theta \cap Der_{M_i}(-logE_i) \subset \omega_i\sm{k+1}$. 
\end{enumerate}
\end{claim}
 
\begin{proof}
Indeed, since $df_1 \wedge \dots \wedge df_n \neq 0$ and $k<n$, the $\mathcal{K}$-monomial foliated Darboux data $(M,\omega\sm{k},\mathcal{D},E)$ is non-trivial. Thus, by Theorem $\ref{thm:main2}$ there exists a collection of $\omega\sm{k}$-admissible local blowings-up:
\[
 \Phi_i : (M_i,\omega_i\sm{k},\mathcal{D}_i,E_i) \rightarrow (M,\omega\sm{k},\mathcal{D}, E) 
\]
such that, for every point $q_i$ in the pre-image of $q$, the invariant $\nu(q_i,\omega_i\sm{k},\mathcal{D}_i)$ is either zero or one. So, by Lemma $\ref{lem:invariantzerodone}$ there exists a $\mathcal{K}$-monomial $(m-k-1)$-foliated Darboux data $(M_i,\omega_i\sm{k+1},\mathcal{D}_i,E_i)$ where
\[
 \omega_i\sm{k+1}\cdot \mathcal{O}_{q_i} = \{\partial \in \omega_i\sm{k} \cdot \mathcal{O}_{q_i} ; \partial(f^\ast_{i_0}) \equiv 0\}
\]
for some index $i_0>k$. So, by the compacity of the pre-image of $q$ and Lemma \ref{lem:Gmonloc}, after shrinking $M_i$ if necessary, we can suppose that the singular distribution $\omega_i\sm{k}$ is $\mathcal{K}$-monomial everywhere in $M_i$ and is independent of the point $q_i$. So, apart from re-indexing, we conclude that:
\[
\begin{aligned}
 \omega_i\sm{k+1} &= \{\partial \in \omega_i\sm{k} ; \partial(f^\ast_{k+1}) \equiv 0\}\\
 &= \{\partial \in Der_{M_i}(-logE_i); \partial(f^{\ast}_i)\equiv 0 \text{ for all }i\leq k+1\}
\end{aligned}
\]
which proves the Claim.
\end{proof}

So, we can recursively use the Claim over $(M,\omega\sm{0},\mathcal{D},E)$ in order to get a {complete} collection of local blowings-up:
\[
 \Phi_i : (M_i,\mathcal{D}_i,E_i) \rightarrow (M,\mathcal{D}, E) 
\]
where, for each point $q_i$ in the pre-image of $p$, there exists a trivial $\mathcal{K}$-monomial $(m-n)$-foliated Darboux data $(M_i,\omega\sm{m-n},\mathcal{D}_i,E_i)$ such that:
\[
 \omega\sm{m-n} = \{\partial \in Der_{U_i}(-logE); \partial(f^{\ast}_i)\equiv 0 \text{ for all }i\leq n\}
\]
Note that the strict transform $\theta_i$ of $\theta$ has first integrals in $\mathcal{D}_i$, which implies that $\theta_i \cap Der_{U_i}(-logE_i) \subset \omega\sm{m-n}$. Now, by Lemma \ref{lem:Fi}, given a point $q_i$ in $U_i$ there exists a monomial coordinate system $(\pmb{u},\pmb{w})$ centered at $q_i$ and $n$-monomial first integrals $\umon{B} = (\umon{\beta_1}, \dots, \umon{\beta_n})$ of $\omega\sm{m-n} . \mathcal{O}_{q_i}$ where $\pmb{B}$ is of maximal rank with coefficients in $\mathcal{K}$. Since $\theta_i \cap Der_{U_i}(-logE_i) \subset \omega\sm{m-n} $, the monomials $\umon{B}= (\umon{\beta_1}, \dots, \umon{\beta_n})$ are also first integrals of $\theta_i$ as we wanted to prove. Finally, if $\mathcal{K}=\mathbb{Q}$, apart from applying Lemma \ref{lem:Final}, we can assume that $\pmb{B}$ has coefficients in $\mathbb{N}$ (instead of $\mathbb{Q}$), which finishes the proof.
 
\section{Theorem \ref{thm:main2}: Overview of the proof}
\label{sec:Overview}

In the remainder of the article, we prove Theorem \ref{thm:main2}. We will make the invariant $\nu$ of $(M,{\omega},\mathcal{D},E)$ decrease by a sequence of ${\omega}$-admissible local blowings up which, by Proposition \ref{prop:AdmCenter}, preserve the $\mathcal{K}$-monomiality of the singular distribution ${\omega}$. Our proof of Theorem \ref{thm:main2} has three main steps (as in \cite{AB,Bel2}):

\medskip\noindent
{\bf Step 1.} Reduction of $\nu$ to a finite value.

\medskip
The following proposition \ref{prop:Infinite} will be proved in Section \ref{sec:Infinte}.

\begin{proposition}[Reduction to a finite invariant]
Let $(M,{\omega},\mathcal{D},E)$ be a non-trivial $\mathcal{K}$-monomial foliated Darboux data and $p$ a point where the invariant $\nu(p,{\omega},\mathcal{D})=\infty$. Then, there exists a ${\omega}$-admissible {complete} collection of local blowings-up $\tau_i : (M_i,{\omega}_i,\mathcal{D}_i,\allowbreak E_i) \rightarrow (M,{\omega},\mathcal{D},E)$ such that, for every point $q_i$ in the pre-image of $p$, the invariant $\nu$ is finite, i.e. $\nu(q_i,{\omega}_i,\mathcal{D}_i)<\infty$ .
\label{prop:Infinite}
\end{proposition}

The proof of this result is a consequence of Theorem \ref{thm:RinvI}. Note that a point where the invariant is finite satisfies the conclusion of Lemma \ref{lem:BasicNormalForm}. In this case we will say that $(M,{\omega},\mathcal{D},E)$ is in \textit{Weierstrass-Tschirnhausen form} at $p$.

\medskip\noindent
{\bf Step 2.} Reduction to prepared normal form.

\medskip
The following Proposition \ref{prop:preparation} will be proved in Section \ref{sec:prep}. 

\begin{proposition}[Preparation]\label{prop:preparation}
Let $p$ be a point of $M$ where the invariant $\nu = \nu(p,{\omega},\mathcal{D})$ is finite and bigger than one, i.e $1< \nu < \infty$. Furthermore, suppose that Theorem $\ref{thm:main2}$ is valid for any $\mathcal{K}$-monomial Darboux data $(N, {\omega'}, \mathcal{S},F)$ with $dim \,N <dim \,M$. Then, there exists a ${\omega}$-admissible {complete} collection of local blowings-up $\tau_i : (M_i,{\omega}_i, \mathcal{D}_i, E_i) \rightarrow (M,{\omega}, \mathcal{D}, E)$ such that, for every point $q_i$ in the pre-image of $p$, the foliated Darboux data $(M_i,{\omega}_i,\mathcal{D}_i,E_i)$ has Weierstrass-Tschirnhausen form of lemma \ref{lem:BasicNormalForm} {(with coordinate system $(\pmb{u},v,\pmb{w})=\allowbreak(u_1,\ldots,  u_r,v,w_{r+2},\allowbreak \ldots, w_m)$)} satisfying (apart from re-indexing) the following additional property for $T_1$:
\[
T_1 = v^{\nu} U + \sum^{\nu-2}_{j=1} v^j \pmb{u}^{\pmb{r}_{j}} b_{j}(\pmb{u},\pmb{w}) + b_{0}(\pmb{u},\pmb{w})
\]
where the functions $U$ is a unit and the function $b_{j}$ is either a unit (and $\pmb{r}_{j} \neq 0$) or zero for $j=1,\ldots, \nu-2$. Furthermore, either $b_{0} = 0$ or:
\[
 b_{0}(\pmb{u},\pmb{w}) = \pmb{u}^{\pmb{\beta}} w_{{m}}^\epsilon 
\]
where $\epsilon \in \{0,1\}$. Finally, the blowings-up involved do not increase the value of $\nu$ over any point, i.e. $\nu(q_i,{\omega}_i,\mathcal{D}_i) \leq \nu(q,{\omega},\mathcal{D})$.
\end{proposition}

\begin{remark}
Note that the inductive hypothesis \textit{``Theorem $\ref{thm:main2}$ is true any foliated Darboux data $(N,{\omega'},\mathcal{S},F)$ with $dim \,N < \,dim M$''} is trivially true when $dimM=1$.
\end{remark}
 
When the foliated Darboux data $(M,{\omega},\mathcal{D},E)$ satisfies the thesis of proposition \ref{prop:preparation} at a point $p$, we will say that $p$ is a \textit{prepared} point and that the Weierstrass-Tschirnhausen form in lemma \ref{lem:BasicNormalForm} is prepared at $p$.

\medskip\noindent
{\bf Step 3.} Further admissible blowings-up to decrease the maximal value of the invariant $\nu$.

\medskip
The following proposition \ref{prop:Dropping} will be proved in Section \ref{sec:drop}.

\begin{proposition}\label{prop:Dropping}
Let $p$ be a point of $M$ where the invariant $\nu = \nu(p,{\omega},\mathcal{D})$ is finite and bigger than one, i.e, $1< \nu < \infty$. Furthermore, suppose that $(M,{\omega},\mathcal{D},E)$ satisfies the prepared normal form at $p$ (see Proposition \ref{prop:preparation}). Then, for a small enough neighborhood $M_0$ of $p$, there exists a sequence of ${\omega}$-admissible blowings-up $\pmb{\tau}: (M_r,{\omega}_r,\mathcal{D}_r,E_r) \to (M_0,{\omega}_0,\mathcal{D}_0,E_0)$ such that, for all point $q$ in the pre-image of $p$, the invariant $\nu$ has dropped, i.e. $\nu(q,{\omega}_r,\mathcal{D}_r) < \nu(p,{\omega},\mathcal{D})$.
\end{proposition}

We can now prove the main technical result of this work:

\begin{proof}[Proof of Theorem \ref{thm:main2}]
The proof of the theorem follows from Propositions \ref{prop:Infinite}, \ref{prop:preparation} and \ref{prop:Dropping} by induction on the dimension of $M$ and the maximal value of the invariant $\nu$. Finally, since all blowings-up are ${\omega}$-admissible, by Proposition \ref{prop:AdmCenter} we conclude that the final involutive distributions are $\mathcal{K}$-monomial.
\end{proof}

\section{Theorem \ref{thm:main2}: Dropping to a finite invariant}
\label{sec:Infinte}

We follow the notation of section \ref{sec:Overview} and we prove Proposition \ref{prop:Infinite} in this section. By Lemma \ref{lem:eqBasica}, we can suppose that equation \ref{eq:Basica} is satisfied. Let $\{\partial_1,\dots,\partial_d\}$ be a monomial system of generators of ${\omega}$ and $(\pmb{u},\pmb{w})$ a monomial coordinate system at $p$. We prove the result by strong induction on the number of singular vector-fields in $\{\partial_1,\dots,\partial_d\}$.\\
\\
\noindent\emph{Base Step:} Suppose that all vector-fields $\partial_i$ are regular. By the definition of monomial coordinate system, this implies that $(\pmb{u},\pmb{w}) = (u_1, \dots, u_{m-d},w_{m-d+1},\dots, w_m)$ and we can assume that $\partial_j = \partial_{w_{k}}$ with $k=m-d+j$. So, let us consider the Taylor expansion of $T_i$ over $p$:
\[
 T_i = \sum_{\lambda} \wmon{\lambda}T_{i,\lambda}(\pmb{u}) 
\]
where $T_{i,0}$ is zero (since $\umon{\delta_i}T_{i,0}$ would be a first integral of ${\omega}$). Now, consider the ideal $\mathcal{I}$ generated by the functions:
\[
 \{T_{i,\lambda}(\pmb{u}); \text{ }\forall \text{ }\lambda \text{ and }i \}
\]
Since $I$ is ${\omega}$-invariant, by Theorem $\ref{thm:RinvI}$ there exists a sequence of ${\omega}$-invariant blowings-up:
\[
\tau: (\widetilde{U},\widetilde{{\omega}},\widetilde{\mathcal{D}},\widetilde{E}) \to (U,{\omega},\mathcal{D},E)
\]
that principalize $I$, where $U$ is an open neighborhood of $p$ where $I$ is well-defined. Since the blowings-up are all ${\omega}$ invariant, for each point $q$ in the pre-image of $p$ there exists a coordinate system $(\pmb{x},\pmb{w})$ such that $ \tau^{\ast}(\partial_i) = \partial_{w_{k}}$ and $I^{\ast}$ is generated by a monomial $\xmon{\beta}$. In particular, let $(i_0,\lambda_0)$ be an index such that $T_{i_0,\lambda_0}^{\ast}=\xmon{\beta}$. Thus:
\[
\begin{aligned}
 T_{i_0}^{\ast} &=  \sum_{\lambda} \wmon{\lambda}T_{i_0,\lambda}(\pmb{u})^{\ast}\\
 &= \xmon{\beta}\left[ \pmb{w}^{\lambda_0}U +   \sum_{\lambda \neq \lambda_0} \pmb{w}^{\lambda} \widetilde{T}_{i_0,\lambda}(\pmb{x}) \right]
\end{aligned}
\]
where $U$ is a unit. Note also that $\lambda_0 \neq 0$ (because $T_{i,0} \equiv 0$ for all $i$). So, the invariant $\nu(q,\widetilde{{\omega}},\widetilde{\mathcal{D}})$ is finite and smaller or equal to $\|\lambda_0\|$.\medskip

\noindent\emph{Induction Step:} Suppose, by strong induction, that the Proposition is true if there are $l_0$ vector-fields in $\{\partial_1,\dots,\partial_d\}$ which are singular with $l_0<l$. We assume that there are $l$ vector-fields over $\{\partial_1,...,\partial_d\}$ that are singular. So, we can rename this set as $\{Y_1, \dots, Y_l,Z_{l+1},\dots Z_d\}$, where the vector-fields $Y_i$ are all singular and $Z_i $ are regular vector-fields. By the definition of monomial coordinate system, we have $(\pmb{u},\pmb{w}) = (u_1, \dots, u_{m-d+l},w_{m-d+l+1},\dots, w_m)$ and we can assume that $Z_j = \partial_{w_{l}}$ with $l=m-d+j$, and $Y_j = \sum \alpha_{i,j}u_k \partial_{u_k}$ for coefficients $\alpha_{i,j} \in \mathcal{K}$. Now, let us consider the Taylor expansion of $T_i$ over $p$:
 \[
 T_i = \sum_{\lambda} \wmon{\lambda}  T_{i,\lambda}(\pmb{u}) 
\]
 Now, note that, given any monomial $\umon{\gamma}$:
 \[
  Y_j(\umon{\gamma}) = K_{j,\gamma}\umon{\gamma}
 \]
where $ K_{j,\gamma}$ is a constant in $\mathcal{K}$. Let $K_{\gamma}$ denote the vector $(K_{1,\gamma}, \dots, K_{l,\gamma})$. In this case, we have a notion of eigenvector associated to the vector-fields $Y_j$:
\[
 T_{i,\lambda}(\pmb{u}) = \sum_{K} T_{i,\lambda,K}(\pmb{u}) 
\]
where all monomials $\umon{\gamma}$ in the expansion of $T_{i,\lambda,K}$ are such that $K_{\gamma}=K$. So, we can write:
\[
 T_i = \sum_{\lambda} \wmon{\lambda} \sum_{K} T_{i,\lambda,K}(\pmb{u}) 
\]
Now, let $I$ be the ideal generated by the functions:
\[
 \{T_{i,\lambda,K}; \text{ }\forall \text{ }i, \text{ } \lambda \text{ and }K \}
\]
Since $I$ is ${\omega}$-invariant, by Theorem $\ref{thm:RinvI}$ there exists a sequence of ${\omega}$-invariant blowings-up:
\[
\tau: (\widetilde{U},\widetilde{{\omega}},\widetilde{\mathcal{D}},\widetilde{E}) \to (U,{\omega},\mathcal{D},E)
\]
that principalize $I$, where $U$ is an open neighborhood of $p$ where $I$ is well-defined. Since the blowings-up are all ${\omega}$-invariant, for each point $q$ in the pre-image of $p$ there exists a coordinate system $(\pmb{x},\pmb{w})$ such that $ \tau^{\ast}(Z_j) = \partial_{w_{k}}$ and $I^{\ast}$ is generated by a monomial $\xmon{\beta}$. In particular, the number of generators of $\widetilde{{\omega}}\cdot \mathcal{O}_q$ which are singular must be smaller or equal than $l$. If the number of singular generators of $q$ is strictly smaller than $l$, we can apply the strong induction hypothesis to obtain a ${\omega}$-admissible {complete} collection of local blowings-up over $q$ so that invariant decreases to a finite value in the pre-image of a  neighborhood of $q$.\\
\\
So, let us assume that there exists $l$ singular vector-fields in $\widetilde{{\omega}} \cdot \mathcal{O}_q$. In particular, these vector-fields should be generated by $Y_j^{\ast}$ (since $Z_j^{\ast}$ are regular). Moreover, there exists an index $(i_0,\lambda_0,K_0)$ such that $T_{i_0,\lambda_0,K_0}^{\ast}$ is a generator of $I^{\ast}$, i.e $T_{i_0,\lambda_0,K_0}^{\ast} = \xmon{\beta}W$, where $W$ is a unit. Then:
\[
 \begin{aligned}
  T_{i_0}^{\ast} &= \sum_{\lambda} \wmon{\lambda}  \sum_{K} T_{i_0,\lambda,K}(\pmb{u})^{\ast}\\
 &= \xmon{\beta}\left[ \pmb{w}^{\lambda_0}\left(W + \sum_{K \neq K_0}\widetilde{T}_{i_0,\lambda_0,K}\right) +   \sum_{\lambda \neq \lambda_0} \pmb{w}^{\lambda} \widetilde{T}_{i_0,\lambda}(\pmb{x}) \right]
 \end{aligned}
\]
We claim that all functions $\widetilde{T}_{i_0,\lambda_0,K}$ with $K\neq K_0$ are \textit{not} unities, which implies that
\[
 W + \sum_{K \neq K_0}\widetilde{T}_{i_0,\lambda_0,K}
\]
is a unit and the invariant $\nu(q,\widetilde{{\omega}},\widetilde{\mathcal{D}})$ is smaller or equal than $\|\lambda_0\|$ (even if $\|\lambda_0\|=0$ since $\umon{\delta_{i}} T_{i,\lambda,K}$ are non-zero eigen-vectors of ${\omega}$). Indeed, let us assume by contradiction that $\widetilde{T}_{i_0,\lambda_0,K}$ is a unit for some $K \neq K_0$. In one hand, this implies that:
\[
 T_{i_0,\lambda_0,K}^{\ast} = \xmon{\beta}V
\]
where $V$ is a unit. By another hand, since $K\neq K_0$, there exists $j_0$ such that the $j_0$ entry of $K$ and $K_0$ are different. Now:
\[
 \begin{aligned}
   Y_{j_0}^{\ast}(T_{i_0,\lambda_0,K_0}^{\ast}) &= Y_{j_0}^{\ast}(\xmon{\beta}W) = <K_{0},e_{j_0}> \xmon{\beta}W \text{ and}\\
Y_{j_0}^{\ast}(T_{i_0,\lambda_0,K}^{\ast}) &= Y_{j_0}^{\ast}(\xmon{\beta}V)= <K,e_{j_0}> \xmon{\beta}V
 \end{aligned}
\]
which implies that:
\[
 \begin{aligned}
   Y_{j_0}^{\ast}(\xmon{\beta}) &= \xmon{\beta} \left(<K_{0},e_{j_0}> +  \frac{Y_{j_0}^{\ast}(W)}{W}  \right) \text{ and}\\
Y_{j_0}^{\ast}(\xmon{\beta})&= \xmon{\beta} \left(<K,e_{j_0}> +  \frac{Y_{j_0}^{\ast}(V)}{V}  \right) 
 \end{aligned}
\]
But, since $Y_{j_0}^{\ast}$ is singular:
\[
 <K_{0},e_{j_0}> +  \frac{Y_{j_0}^{\ast}(W)}{W} \neq <K,e_{j_0}> +  \frac{Y_{j_0}^{\ast}(V)}{V}  
\]
which is a contradiction. The invariant is finite in an open neighborhood of $q$.

\section{Theorem \ref{thm:main2}: Prepared normal form}
\label{sec:prep}

We follow the notation of Lemma \ref{lem:BasicNormalForm} and section \ref{sec:Overview} and we prove Proposition \ref{prop:preparation} in this section. By Lemma \ref{lem:BasicNormalForm}, the $\mathcal{K}$-monomial foliated Darboux data $(M,{\omega},\mathcal{D},E)$ satisfies the Weierstrass-Tschirnhausen form at $p$, i.e. there exists a monomial coordinate system $(\pmb{u},v,\pmb{w})$ of $p$ such that the functions $T_i$ are given by \eqref{eq:basicnormaform} and the vector-field $\partial_{v}$ belongs to ${\omega}_{p}$ (in particular, we assume that $\partial_v^{\nu}T_1$ is a unit). The main idea of the proof is to modify the coefficients $a_{1,j}$ without changing the $v$-coordinate. This is obtained through two steps, where all blowings-up are ${\omega}$-admissible and $\partial_v$-invariant.\\
\\
\textit{First Step:} Let us perform a ${\omega}$-admissible collection of local blowings-up to get all necessary conditions over the coefficients $a_{1,j}$ with $j>0$. Indeed, let $\pi: M_0 \rightarrow N$ be the projection map given by $\pi(\pmb{u},v,\pmb{w}) = (\pmb{u},\pmb{w})$, where $M_0$ is a small enough neighborhood of $p$, and let $\mathcal{J}$ be the principal ideal sheaf generated by the product of all non-zero $a_{1,j}$ with $j>0$. Then, there exists a $d-1$ foliated ideal sheaf $(N,{\omega'},\mathcal{J},F)$ such that:
 
\begin{itemize}
 \item The singular distribution ${\omega}$ is generated by $\pi^{\ast}{\omega'}$;
 \item The inverse image of $F$ is equal to $E \cap M_0$.
\end{itemize}
 
Now, by Theorem \ref{thm:RI} there exists a ${\omega'}$-admissible {complete} collection of local blowings-up
\[
\sigma_i : (N_i,{\omega_i'},\mathcal{J}_i,F_i) \rightarrow (N,{\omega'},\mathcal{J},F) 
\]
 such that the ideal sheaf $\mathcal{J}_i$ is monomial i.e. $\sigma_i^{\ast} \mathcal{J}$ is a principal ideal sheaf with support contained in $F_i$. We can extend $\sigma_i$ to blowings-up at $M_0$ by taking the product of the centers of $\tau_i$ by the $v$-coordinate:
\[
 \tau_i\sm{1}: (M_i\sm{1},{\omega}_i\sm{1},\mathcal{D}_i\sm{1},E_i\sm{1}) \to (M_0,{\omega}_0,\mathcal{D}_0,E_0)
\]
where all centers have SNC with the exceptional divisor and are invariant by the $v$-coordinate i.e. all centers are $\partial_{v}$-invariant. Moreover, since all centers of $\sigma_i$ are ${\omega'}$-admissible, we conclude that all centers of $\tau_i\sm{1}$ are ${\omega}$-admissible.\\
\\
Now, consider a point $q_i$ in the pre-image of $p$ by $\tau_i\sm{1}$ and let $(\pmb{u}\sm{1},v\sm{1},\pmb{w}\sm{1})$ be a coordinate system at $q_i$ such that $\tau_i\sm{1}^{\ast}v=v\sm{1}$. Since the pull-back $(\tau_i\sm{1} \circ \pi)^{\ast} \mathcal{J}$ is a principal ideal sheaf, we conclude that:
\[
 \begin{aligned}
  T_1 &= v\sm{1}^{\nu} U + \sum^{\nu-2}_{j=1} v\sm{1}^j \pmb{u}\sm{1}^{r_{j}\sm{1}} b_{1}\sm{1} + b_{0}\sm{1} \text{ where $U$ is a unit}
 \end{aligned}
\]
where the functions $b_{j}\sm{1}$ are either zero or units for $j>0$ and the monomials $\pmb{u}\sm{1}^{r_{j}\sm{1}}$ have support in the exceptional divisor $E_i\sm{1}$. Note that $\partial_{v\sm{1}}$ belongs to $ {\omega}_i \cdot \mathcal{O}_{q_i}$ and, in particular, that $\nu(q_i,{\omega}_i,\mathcal{D}_i) \leq \nu(p,{\omega},\mathcal{D})$.\\
\\
\textit{Second Step:} We now perform a ${\omega}$-admissible {complete} collection of local blowings-up to get all necessary conditions over the coefficients $b_{0}\sm{1}$. Indeed, at each point $q_i$ in the pre-image of $p$, apart from taking smaller varieties $M_i\sm{1}$, there exists a projection map $\pi: M_i\sm{1} \rightarrow N_i\sm{1}$ given by $\pi(\pmb{u}\sm{1},v\sm{1},\pmb{w}\sm{1}) = (\pmb{u}\sm{1},\pmb{w}\sm{1})$. Then, there exists a $d-1$ foliated Darboux data $(N_i\sm{1},{\omega_i'\sm{1}},\mathcal{S}_i\sm{1},F_i\sm{1})$ such that:
 
\begin{itemize}
 \item The singular distribution ${\omega}_i\sm{1}$ is generated by $\pi^{\ast}{\omega_i'\sm{1}}$;
 \item The inverse image of $F_i\sm{1}$ is equal to $E_i\sm{1}$;
 \item The $\mathcal{K}$-Darboux data $\mathcal{S}_i\sm{1}$ is generated by $
  f_1|_{\{v\sm{1}=0\}}$.
\end{itemize}
Note that if $b_{0}\sm{1} = 0$ we are done. Otherwise, the foliated Darboux data $(N_i\sm{1},{\omega_i'\sm{1}},\allowbreak\mathcal{S}_i\sm{1},F_i\sm{1})$ is not trivial and, since $dim \, N_i\sm{1}<dim \, M_i\sm{1}$, we can apply Theorem $\ref{thm:main2}$ to $(N_i\sm{1},{\omega_i'\sm{1}},\mathcal{S}_i\sm{1},F_i\sm{1})$ in order to obtain a ${\omega_i'\sm{1}}$-admissible {complete} collection of local blowings-up
\[
 \sigma_{i,j}\sm{2}: (N_{i,j}\sm{2},{\omega_{i,j}'\sm{2}},\mathcal{S}_{i,j}\sm{2},F_{i,j}\sm{2}) \rightarrow (N_i\sm{1},{\omega_i'\sm{1}},\mathcal{S}_i\sm{1},F_i\sm{1}) 
\]
such that, the invariant $\nu$ calculated for $(N_{i,j}\sm{2},{\omega_{i,j}'\sm{2}},\mathcal{S}_{i,j}\sm{2},F_{i,j}\sm{2})$ is zero or one at every point. Furthermore, by Lemma $\ref{lem:invariantzerodone}$, at each point in the pre-image of $q_i$, there exists a coordinate $(\pmb{u}\sm{2},\pmb{w}\sm{2})$ such that:
\begin{equation}
\label{eq:transform2}
\begin{aligned}
 \sigma_{i,j}\sm{2}^{\ast}\left[ f_1|_{\{v\sm{1}=0\}} \right] &=  g_{1}\sm{2} + \pmb{u}\sm{2}^{\pmb{\widetilde{\beta}}}w_{{m}}\sm{2}^{\epsilon}
\end{aligned}
\end{equation}
where $g_{1}\sm{2}$ is a first integral of ${\omega_{i,j}'\sm{2}}$, $\pmb{u}\sm{2}^{\pmb{\widetilde{\beta}}}w_{{m}}\sm{2}^{\epsilon}$ is not a first integral of ${\omega_{i,j}'\sm{2}}$ and $\epsilon \in\{0,1\}$. We can extend $\sigma_{i,j}\sm{2}$ to blowings-up at $M_i\sm{1}$ by taking the product of the centers of $\tau_{i,j}\sm{2}$ by the $v$-coordinate:
\[
 \tau_{i,j}\sm{2}: (M_{i,j}\sm{2},{\omega}_{i,j}\sm{2},\mathcal{D}_{i,j}\sm{2},E_{i,j}\sm{2}) \to (M_i\sm{1},{\omega}_i\sm{1},\mathcal{D}_i\sm{1},E_i\sm{1})
\]
where all centers have SNC with the exceptional divisor and are invariant by the $v$-coordinate i.e. all centers are $\partial_{v}$-invariant. Moreover, since all centers of $\sigma_{i,j}\sm{2}$ are ${\omega'}$-admissible, we conclude that all centers of $\tau_{i,j}\sm{2}$ are ${\omega}$-admissible.\\
\\
Now, consider a point $q_{i,j}$ in the pre-image of $q_i$ and let $(\pmb{u}\sm{2},v\sm{2},\pmb{w}\sm{2})$ be a monomial coordinate system of $q_{i,j}$ such that $\tau_{i,j}\sm{2}^{\ast}v\sm{1}=v\sm{2}$. By equation $\eqref{eq:transform2}$
\[
\begin{aligned}
 \tau_{i,j}\sm{2}^{\ast}\left[ g_{1} + \umon{\delta_{1}}a_{1,0} \right] &=  g_{1}\sm{2} + \pmb{u}\sm{2}^{\pmb{\widetilde{\beta}}}w_{{m}}\sm{2}^{\epsilon} 
\end{aligned}
  \]
  where $g_{1}\sm{2}$ is a first integral of ${\omega}_{i,j}\sm{2}$, $ \pmb{u}\sm{2}^{\pmb{\widetilde{\beta}}}w_{{m}}\sm{2}^{\epsilon}$ is not a first integral of ${\omega}_{i,j}\sm{2}$ and $\epsilon \in\{0,1\}$. Furthermore, since all blowings-up have SNC with the exceptional divisor, we conclude that:
\[
 \begin{aligned}
  T_1 &= v\sm{2}^{\nu} U + \sum^{\nu-2}_{j=1} v\sm{2}^j \pmb{u}\sm{2}^{r_{j}\sm{2}} c_{j}\sm{2} + c_{1,0}\sm{2} \text{ where $U$ is a unit}
 \end{aligned}
\]
where the functions $c_{j}\sm{2}$ are either zero or units for $j>0$, the monomials $\pmb{u}\sm{2}^{r_{j}\sm{2}}$ have support in the exceptional divisor $E_{i,j}\sm{2}$ and
 \[
\begin{aligned}
c_{0} &= \pmb{u}\sm{2}^{\pmb{\beta}}w_{{m}}\sm{2}^{\epsilon} 
\end{aligned}
  \]
where $\epsilon \in \{0,1\}$ and $\pmb{\beta}$ is equal to the multi-index $\pmb{\widetilde{\beta}}$ minus the multi-index that corresponds to the pull-back of $\umon{\delta}$. To finish, note that $\partial_{v\sm{2}}$ belongs to $ {\omega}_{i,j}\sm{2} . \mathcal{O}_{q_i}$ and, in particular, that $\nu(q_{i,j},{\omega}_{i,j},\mathcal{D}_{i,j}) \leq \nu(p,{\omega},\mathcal{D})$.

\section{Theorem \ref{thm:main2}: Dropping the invariant $\nu$}
\label{sec:drop}

We follow the notation of section \ref{sec:Overview} and we prove proposition \ref{prop:Dropping} in this section. We start by a couple of preliminary results about the combinatorial blowings-up.

\subsection{Combinatorial blowings up}
 
\begin{definition}[Sequence of combinatorial blowings-up]\label{def:combblowingsup}
Given a divisor $E$ in $M$, we say that $\tau:\widetilde{M}\to M$ is a sequence of combinatorial blowings-up (with respect to $E$) if $\tau$ is a composition of blowings-up with centers that are strata of the divisor $E$ and its total transforms.
\end{definition}

Consider a $\mathcal{K}$-monomial foliated manifold $(M,{\omega},E)$ and suppose that $(\pmb{u},v,\pmb{w})$ is a globally defined monomial coordinate system centered at a point $p$, where the vector-field $\partial_{v}$ belongs to ${\omega}$. We remark that, by Lemma \ref{lem:Fi} there exists a collection of monomials $\umon{B} = ( \umon{\beta_1}, \dots, \umon{\beta_{m-d}})$ such that a vector field $\partial \in Der_M(-logE)$ belongs to ${\omega}$ if and only if $\partial(\umon{\beta_i})\equiv 0$ for all $i$.

Consider a sequence of combinatorial blowings-up $\pmb{\tau}:(\widetilde{M},\widetilde{{\omega}},\widetilde{E})\to (M,{\omega},E)$ with respect to the declared exceptional divisor $F=\{u_1\dotsm u_r \cdot v =0\}$. Note that such a sequence is ${\omega}$-admissible and, by Proposition \ref{prop:AdmCenter}, the transform $\widetilde{{\omega}}$ is $\mathcal{K}$-monomial. Moreover, we can cover $\widetilde{M}$ by affine charts with a coordinate system $(\pmb{x},\pmb{w})$ satisfying:
\begin{equation}
\label{eq:blowingsup}
\begin{aligned}
u_j &= x_1^{a_{j,1}} \cdots  x_{r+1}^{a_{j,r+1}} \\
v\phantom{_j} &= x_1^{\alpha_{1}} \cdots  x_{r+1}^{\alpha_{r+1}}\\
w_i &= w_i
\end{aligned}
\end{equation}
(where $\alpha_{i,j} \in \mathbb{N}$) that we denote by:
\[
(\pmb{u},v,\pmb{w}) = (\xmon{\mathcal{A}},\pmb{w}) = (\xmon{A},\xmon{\alpha},\pmb{w})
\]
where $\pmb{\mathcal{A}}$ is a $(r+1)$-square matrix $\begin{bmatrix}
                            \pmb{A}\\
                            \pmb{\alpha}
                           \end{bmatrix}$ given by:
\[
\pmb{A} =
\begin{bmatrix}
 a_{1,1} & \ldots  & a_{1,r+1} \\
 \vdots &  \ddots & \vdots \\
a_{r,1} & \ldots  & a_{r,r+1} 
\end{bmatrix}
\text{ and } \pmb{\alpha} = \begin{bmatrix}
                                    \alpha_1 & \dots & \alpha_{r+1}
                                   \end{bmatrix}
\]
Note that $(\pmb{x},\pmb{w})$ is a monomial coordinate system since, by Lemma \ref{lem:CompForm}:
\[
 \tau^{\ast}\umon{B} = \xmon{B A}
\]
is a system of first integrals of $\widetilde{{\omega}}$. Now, let $q$ be {a} point in this affine chart contained in the pre-image of $p$ (recall that $p$ is the origin of the original coordinate system). Apart from re-indexing, we can suppose that $q$ has coordinates $(0,\pmb{\xi},0) = (0,\ldots, 0, \xi_{t+1},\ldots, \xi_{r+1},0, \ldots, 0)$ with $\xi_i \neq 0$. Furthermore, apart from making changes of the form $x_i=-x_i$, we can suppose that $\xi_i<0$ whenever $\xi_i \in \mathbb{R}$. We consider the coordinate system $(\pmb{x}, \pmb{y},\pmb{w})= (x_1,\ldots, x_t,y_{y+1},\ldots, y_{r+1},w_{r+2},\ldots, \allowbreak w_m)$ centered at $q$ where:
\[
y_i = x_i -\xi_{i}
\]
Note that $t \neq 0$ since $q$ is in the pre-image og $p$. We have a decomposition of the matrix $\pmb{\mathcal{A}}$
\[
\pmb{\mathcal{A}} = 
\begin{bmatrix}
 \pmb{A_1}&\pmb{A_2} \\
 \pmb{\alpha_1}&\pmb{\alpha_2}
\end{bmatrix}
\]
where $\pmb{A_1}$ is a $r \times t$ matrix, $\pmb{A_2}$ is a $r\times (r+1-t)$ matrix, $\pmb{\alpha_1}$ is a $1\times t$ matrix and $\pmb{\alpha_2}$ is a $1\times (r-t+1)$ matrix. We remark that, since $q$ is a point on the exceptional divisor $\widetilde{E}$, there exists at least one $u_i$ such that $\pmb{\tau}^{\ast} u_i (q) = 0$, which implies that $\pmb{A_1}$ has to be a non-zero matrix. We now divide our study depending on the rank of $\pmb{A_1}$:

\begin{lemma}[Case 1]\label{lem:claim1}
Assume $\pmb{A_1}$ has maximal rank. Then, there exists a monomial coordinate system $(\pmb{x},\pmb{y},z,\pmb{w}) = (x_1,\ldots, x_t,y_{t+1},\ldots ,\allowbreak y_{r},z,\pmb{w})$ centered at $q$ such that 
\begin{equation}\label{eq:claim1}
\begin{aligned}
\pmb{u}&=\pmb{x}^{\pmb{A}_1}(\pmb{y}-\pmb{\xi})^{\pmb{\Lambda}}\\
v&=\xmon{\alpha_1}(z-\zeta)\\
\pmb{w}&=\pmb{w}
\end{aligned}
\end{equation}
where $\zeta \neq 0$, $\xi_{j} \neq 0$ for all $j$ and the matrix $\pmb{\Lambda}= (\lambda_{i,j})$ of exponents has maximal rank, with $\lambda_{i,j}\in\mathcal{K}$ (in particular, $\partial_z$ is contained in $\widetilde{{\omega}}\cdot \mathcal{O}_q$). Moreover, if $\umon{\gamma}$ is \textbf{not} a first integral of ${\omega}\cdot \mathcal{O}_p$, then its total transform $\umon{\gamma} = \xmon{\widetilde{\gamma}}U$, where $U$ is a unit, satisfies one of the following:
\begin{itemize}
 \item Either the monomial $\xmon{\widetilde{\gamma}}$ is \textbf{not} a first integral of $\widetilde{{\omega}}$;
 \item Or, there exists a regular vector-field $\partial_{y_i} \in \widetilde{{\omega}}\cdot \mathcal{O}_q$ such that $\partial_{y_i} U$ is a unit.
\end{itemize}
 
\end{lemma}

\begin{lemma}[Case 2]\label{lem:claim2}
Assume that $\pmb{A_1}$ does not have maximal rank. Then, there exists a monomial coordinate system $(\pmb{x},\pmb{y},\pmb{w}) = (x_1,\ldots, x_t,y_{t+1}, \allowbreak \ldots  ,  y_{r+1},\allowbreak \pmb{w})$ centered at $q$ such that
 
\begin{equation}
\label{eq:claim2}
\begin{aligned}
\pmb{u}&=\pmb{x}^{A_1}(\pmb{y}-\pmb{\xi})^{\Lambda}\\
v&=\xmon{\alpha_1}\\
\pmb{w}&=\pmb{w}
\end{aligned}
 \end{equation}
 
where $\xi_{j} \neq 0$ for all $j$, the matrix $\pmb{\Lambda}= (\lambda_{i,j})$ is of maximal rank with entries in $\mathcal{K}$, and $\pmb{\alpha_1}$ doesn't belong to the span of the rows of $\pmb{A_1}$. Moreover, if $\umon{\gamma}$ is \textbf{not} a first integral of ${\omega}\cdot \mathcal{O}_p$, then its total transform $\umon{\gamma} = \xmon{\widetilde{\gamma}}U$, where $U$ is a unit, satisfies one of the following:
\begin{itemize}
 \item Either the monomial $\xmon{\widetilde{\gamma}}$ is \textbf{not} a first integral of $\widetilde{{\omega}}$;
 \item Or, there exists a regular vector-field $\partial_{y_i} \in \widetilde{{\omega}}\cdot \mathcal{O}_q$ such that $\partial_{y_i} U$ is a unit.
\end{itemize}
\end{lemma}
 
\begin{proof}[Proof of Lemma \ref{lem:claim1}]
By hypothesis, apart from re-indexing of the $u_j$'s, we can write 
\[
\pmb{A_1}=\begin{bmatrix}\pmb{A_1^{'}}\\\pmb{A_1^{''}}\end{bmatrix},\ \pmb{A_2}=\begin{bmatrix}\pmb{A_2^{'}}\\\pmb{A_2^{''}}\end{bmatrix},
\]
where $\text{det}(\pmb{A_1^{'}})\neq0$, and $\pmb{A_1^{'}}$ and $\pmb{A_2^{'}}$ have the same height. So, we can write equations \eqref{eq:blowingsup} in the compact form
\begin{align}\notag
\pmb{u}'\phantom{'}&=\pmb{x}^{\pmb{A_1^{'}}}(\pmb{y}-\pmb{\xi})^{\pmb{A_2^{'}}}\\ \notag
\pmb{u}''&=\pmb{x}^{\pmb{A_1^{''}}}(\pmb{y}-\pmb{\xi})^{\pmb{A_2^{''}}}\\ \notag
v\phantom{''}&=\pmb{x}^{\pmb{\alpha_1}}(\pmb{y}-\pmb{\xi})^{\pmb{\alpha_2}}
\end{align}

\noindent\emph{First change of coordinates:} 
\begin{align*}
\pmb{x}\sm{1}&=\pmb{x}\cdot (\pmb{y}-\pmb{\xi})^{(\pmb{A_1^{'}})^{-1}\pmb{A_2^{'}}}\\
\pmb{y}\sm{1}&=\pmb{y}
\end{align*}
After this change of coordinates we get (using Lemma \ref{lem:CompForm})
\begin{equation}\label{eq:afterchangecoord1}
\begin{aligned}
\pmb{u}'\phantom{'}&=\pmb{x}\sm{1}^{\pmb{A_1^{'}}}\\
\pmb{u}''&=\pmb{x}\sm{1}^{\pmb{A_1^{''}}}(\pmb{y}\sm{1}-\pmb{\xi})^{\pmb{\Lambda}\sm{1}}\\
v\phantom{''}&=\pmb{x}\sm{1}^{\pmb{\alpha_1}}(\pmb{y}\sm{1}-\pmb{\xi})^{\pmb{\lambda}\sm{1}}
\end{aligned}
\end{equation}
and the square matrix $\mathcal{L}:=\begin{bmatrix}\pmb{\Lambda}\sm{1}\\ \pmb{\lambda}\sm{1}\end{bmatrix}:=\begin{bmatrix}\pmb{A_2^{''}}-\pmb{A_1^{''}}(\pmb{A_1^{'}})^{-1}\pmb{A_2^{'}}\\ \pmb{\alpha_2}-\pmb{\alpha_1}(\pmb{A_1^{'}})^{-1}\pmb{A_2^{'}}\end{bmatrix}$ has determinant different from zero because the full matrix of exponents in \eqref{eq:afterchangecoord1} is obtained from $\pmb{\mathcal{A}}$ by a sequence of column elementary transformations. Note also that the entries of $(\pmb{A_1^{'}})^{-1}\pmb{A_2^{'}}$ are rational numbers, not necessarily integers.
\medskip

\noindent\emph{Second change of coordinates:} After re-indexing the $y_i\sm{1}-\xi_i$ we can assume that the elements of the diagonal of $\mathcal{L}$ are different from zero. Thus, we consider
\begin{equation}
\begin{aligned}
\pmb{y}\sm{2}-\pmb{\xi}\sm{2}&=(\pmb{y}\sm{1}-\pmb{\xi})^{\pmb{\Lambda}\sm{1}}\\
z\sm{2}-\zeta\sm{2}&=(\pmb{{y}}\sm{1}-\pmb{\xi})^{\pmb{\lambda}\sm{1}}\\
\pmb{x}\sm{2}&=\pmb{x}\sm{1}
\end{aligned}
\end{equation}
so to get
\begin{equation}\label{eq:afterchangecoord2}
\begin{aligned}
\pmb{u}'\phantom{'}&=\pmb{x}\sm{2}^{\pmb{A_1^{'}}}\\
\pmb{u}''&=\pmb{x}\sm{2}^{\pmb{A_1^{''}}}(\pmb{y}\sm{2}-\pmb{\xi}\sm{2})^{\pmb{Id}}\\ \notag
v\phantom{''}&=\pmb{x}\sm{2}^{\pmb{\alpha_1}}(z\sm{2}-\zeta\sm{2})
\end{aligned}
\end{equation}
where $\pmb{Id}$ is the identity matrix.
\medskip

\noindent\emph{Third change of coordinates:} We need to guarantee that the coordinate system is monomial. Consider a complete system of first integrals $\umon{B}$ and note that the matrix $\pmb{B}$ can be written as:
\[
\pmb{B} = \begin{bmatrix}
      \pmb{B_1 }& \pmb{B_2}
     \end{bmatrix}
\]
where $\pmb{B_2}$ is a $r\times (r-t)$ matrix. With this notation, by Lemma \ref{lem:CompForm}:
\[
 \umon{B} = \pmb{x}\sm{2}^{\pmb{B A_1}}(\pmb{y}\sm{2} - \pmb{\xi}\sm{2})^{\pmb{B_2}}  =\pmb{x}\sm{2}^{\pmb{C_1}}(\pmb{y}\sm{2} - \pmb{\xi}\sm{2})^{\pmb{C_2}}
\]
where $\pmb{C_1} =\pmb{B A_1}$ is a non-zero matrix (since $\pmb{B}$ and $\pmb{A_1}$ are of maximal rank) and $\pmb{C_2} =\pmb{B_2}$. Now, we perform a change of coordinates similar with the one given in Lemma \ref{lem:Gmonloc} in order to obtain a monomial coordinate system. To that end, consider:
\[
 \pmb{C} = \begin{bmatrix}
 \pmb{C^{'}_1}&\pmb{C^{'}_2} \\
 \pmb{C^{''}_1}&\pmb{C^{''}_2}
\end{bmatrix}
\]
where $ \pmb{C_1} = \begin{bmatrix} \pmb{C^{'}_1} \\ \pmb{C^{''}_1}\end{bmatrix}$ and the rank of $\pmb{C^{'}_1}$ is maximal and equal to the rank of $\pmb{C_1}$. So, there exists a change of coordinates $(\pmb{x}\sm{3},\pmb{y}\sm{3},z\sm{3},\pmb{w}\sm{3})$ where $z\sm{3} = z\sm{2}$, such that:
\[
 \umon{B} = \pmb{x}\sm{3}^{\pmb{D_1}}(\pmb{y}\sm{3}-\pmb{\xi}\sm{3})^{\pmb{D_2}}
\]
where:
\[
 \pmb{D} = \begin{bmatrix}
                   \pmb{D_1} & \pmb{D_2}
                  \end{bmatrix} = \begin{bmatrix}
 \pmb{D^{'}_1}&\pmb{D^{'}_2} \\
 \pmb{D^{''}_1}&\pmb{D^{''}_2}
\end{bmatrix} = \begin{bmatrix}
 \pmb{C^{'}_1}& 0  \\
 \pmb{C^{''}_1}&\pmb{\Delta}
\end{bmatrix}
\]
where $\pmb{\Delta}$ is a maximal rank matrix with coefficients in $\mathcal{K}$. This implies that the collection $(\pmb{x}\sm{3}^{\pmb{D_1^{'}}},\pmb{x}\sm{3}^{\pmb{D_1^{''}}}(\pmb{y}\sm{3}-\pmb{\xi}\sm{3})^{\pmb{\Delta}} )$ is a collection of first integrals of ${\omega}\cdot \mathcal{O}_q$.\\
\\
\noindent\emph{Fourth change of coordinates:} Let $s$ be the rank of $\pmb{\Delta}$. Then, apart from re-ordering the $\pmb{y}\sm{3}$ coordinates, there exists a coordinate system $(\pmb{x}\sm{4},\pmb{y}\sm{4},\pmb{v}\sm{4},z\sm{4}, \allowbreak \pmb{w}\sm{4})$ where $\pmb{x}\sm{4} = \pmb{x}\sm{3}$, $z\sm{4}=z\sm{3}$, $\pmb{w}\sm{4}=\pmb{w}\sm{3}$ and $\pmb{v}\sm{4} = (y_{t+s+1}\sm{3},\dots, y_{r+1}\sm{3})$ such that:
\[
 (\pmb{y}\sm{3}-\pmb{\xi}\sm{3})^{\pmb{\Delta}} = \pmb{y}\sm{4}-\pmb{\xi}_y\sm{4}
\]
where $\pmb{y}\sm{4} = (y_{t+1}\sm{4},\dots, y_{s}\sm{4})$,
which implies that the monomial functions 
\[
(\pmb{x}\sm{4}^{\pmb{D_1^{'}}},\pmb{y}\sm{4}) 
\]
are first integrals of ${\omega}\cdot \mathcal{O}_q$, which guarantees that the coordinate system is monomial. Furthermore, since $z\sm{2}=z\sm{4}$ we finally conclude that:
\[
\begin{aligned}
\pmb{u}&=\pmb{x}\sm{4}^{\pmb{A_1}}(\pmb{y}\sm{4}-\pmb{\xi}_y\sm{4})^{\pmb{\Lambda}_y\sm{4}}(\pmb{v}\sm{4}-\pmb{\xi}_v\sm{4})^{\pmb{\Lambda}_v\sm{4}}\\
v&=\pmb{x}\sm{4}^{\pmb{\alpha_1}}(\pmb{y}\sm{4}-\pmb{\xi}_y\sm{4})^{\pmb{\lambda}_y\sm{4}}(\pmb{v}\sm{4}-\pmb{\xi}_v\sm{4})^{\pmb{\lambda}_v\sm{4}} (z\sm{4}-\zeta\sm{4})\\
\pmb{w}&=\pmb{w}
\end{aligned}
\]
where $\pmb{\Lambda}\sm{4} = [\pmb{\Lambda}_y\sm{4},\pmb{\Lambda}_v\sm{4}]$ is a maximal rank matrix whose entries are in $\mathcal{K}$.\\
\\
\noindent\emph{Fifth change of coordinates:} We only need a change in the $z\sm{4}$ so that 
\[
z\sm{5}-\zeta\sm{5} = (\pmb{y}\sm{4}-\pmb{\xi}_y\sm{4})^{\pmb{\lambda}_y\sm{4}}(\pmb{v}\sm{4}-\pmb{\xi}_v\sm{4})^{\pmb{\lambda}_v\sm{4}} (z\sm{4}-\zeta\sm{4})                                                                                                                                                                                                                                                                             \]
which does not change the fact that the coordinate system is monomial. This is the coordinate system of the enunciate of the Lemma.\\
\\
Now, let $\umon{\gamma}$ be a monomial which is \textbf{not} a first integral of ${\omega}\cdot \mathcal{O}_p$, i.e., the multi-index $\pmb{\gamma}$ doesn't belong to the span of the rows of $\pmb{B}$. In this case, in the coordinate system of the enunciate of the Lemma, we have that:
\[
 \umon{\gamma} = \pmb{x}^{\pmb{\gamma}\pmb{A}_1}(\pmb{y}-\pmb{\xi})^{\pmb{\gamma}\pmb{\Lambda}} 
\]
In particular, there exists a vector field {$\widetilde{\partial}$} in ${\omega} \cdot \mathcal{O}_p$ such that $\widetilde{\partial}(\umon{\gamma})\neq 0$. Now, recall that $\pmb{x}^{\pmb{A_1B}}$ are monomial first integrals of ${\omega} \cdot \mathcal{O}_q$. So, either $\pmb{\gamma A_1}$ is not in the subspace generated by the rows of $\pmb{A_1B}$ (and $\pmb{x}^{\pmb{\gamma}\pmb{A}_1}$ is not a first integral), or it is and:
\[
\partial(\pmb{x}^{\pmb{\gamma}\pmb{A}_1})\equiv 0 \text{ for all }\partial \in {\omega}_q
\]
In this case, we conclude that:
\[
{\sigma^{\ast}(\widetilde{\partial})}(\pmb{y}-\pmb{\xi})^{\pmb{\gamma}\pmb{\Lambda}}  \neq 0
\]
which implies that $\pmb{\gamma}\pmb{\Lambda}\neq 0$ {and concludes the Lemma}.
\end{proof}
 
\begin{proof}[Proof of Lemma \ref{lem:claim2}]
We have 
\begin{align}\notag
\umon{B}&=\pmb{x}^{\pmb{B A_1}}(\pmb{y}-\pmb{\xi})^{\pmb{B A_2}}\\ \notag
\pmb{u}\phantom{^B}&=\pmb{x}^{\pmb{A_1}}(\pmb{y}-\pmb{\xi})^{\pmb{A_2}}\\ \notag
v\phantom{^B}&=\pmb{x}^{\pmb{\alpha_1}}(\pmb{y}-\pmb{\xi})^{\pmb{\alpha_2}},
\end{align}
Note that, since $\pmb{\mathcal{A}}$ is of maximal rank but $\pmb{A_1}$ does not have maximal rank, $\pmb{\alpha_1}$ doesn't belong to the span of the rows of $\pmb{A_1}$. Thus, it does not belong to the span of the rows of $\pmb{B A_1}$.\\
\\
\noindent\emph{First change of coordinates:} There exists a coordinate system $(\pmb{x}\sm{1},\pmb{y}\sm{1},\pmb{w}\sm{1})$ where $\pmb{y}\sm{1}=\pmb{y}$ and $\pmb{w}\sm{1}=\pmb{w}$ such that
\begin{align}\notag
\umon{B}&=\pmb{x}\sm{1}^{\pmb{C_1}}(\pmb{y}\sm{1}-\pmb{\xi})^{\pmb{C_2}}\\ \notag
\pmb{u}\phantom{^B}&=\pmb{x}\sm{1}^{\pmb{A_1}}(\pmb{y}\sm{1}-\pmb{\xi})^{\pmb{\Lambda}\sm{1}}\\ \notag
v\phantom{^B}&=\pmb{x}\sm{1}^{\pmb{\alpha_1}}
\end{align}
where $\pmb{C} = \begin{bmatrix}
                         \pmb{C_1}\\ \pmb{C_2}
                        \end{bmatrix}=\begin{bmatrix}
                         \pmb{B A_1}\\ \pmb{C_2}
                        \end{bmatrix}
$ and $\pmb{\Lambda}\sm{1}$ are matrices of maximal rank with coefficients in $\mathbb{Q}$.\\
\\
\noindent\emph{Second change of coordinates:} We need to guarantee that the coordinate system is monomial. To that end, consider:
\[
 \pmb{C} = \begin{bmatrix}
 \pmb{C^{'}_1}&\pmb{C^{'}_2} \\
 \pmb{C^{''}_1}&\pmb{C^{''}_2}
\end{bmatrix}
\]
where $ \pmb{C_1} = \begin{bmatrix} \pmb{C^{'}_1} \\ \pmb{C^{''}_1}\end{bmatrix}$ and the rank of $\pmb{C^{'}_1}$ is maximal and equal to the rank of $\pmb{C_1}$. Since $\pmb{\alpha_1}$ does not belong to the span of the rows of $\pmb{C^{'}_1}$, there exists a change of coordinates $(\pmb{x}\sm{2},\pmb{y}\sm{2},\pmb{w}\sm{2})$ where $v=\pmb{x}\sm{2}^{\pmb{\alpha_1}}$, such that:
\[
\begin{aligned}
\umon{B}&=\pmb{x}\sm{2}^{\pmb{D_1}}(\pmb{y}\sm{2}-\pmb{\xi}\sm{2})^{\pmb{D_2}}\\
\pmb{u}\phantom{^B}&=\pmb{x}\sm{2}^{\pmb{A_1}}(\pmb{y}\sm{2}-\pmb{\xi}\sm{2})^{\pmb{\Lambda}\sm{2}}\\ 
v\phantom{^B}&=\pmb{x}\sm{2}^{\pmb{\alpha_1}}
\end{aligned}
\]
where $\pmb{\Lambda}\sm{2}$ is a maximal rank matrix with entries in $\mathcal{K}$ and
\[
 \pmb{D} = \begin{bmatrix}
                   \pmb{D_1} & \pmb{D_2}
                  \end{bmatrix} = \begin{bmatrix}
 \pmb{D^{'}_1}&\pmb{D^{'}_2} \\
 \pmb{D^{''}_1}&\pmb{D^{''}_2}
\end{bmatrix} = \begin{bmatrix}
 \pmb{C^{'}_1}& 0  \\
 \pmb{C^{''}_1}&\pmb{\Delta}
\end{bmatrix}
\]
where $\pmb{\Delta}$ is a maximal rank matrix with entries in $\mathcal{K}$. This implies that the collection $(\pmb{x}\sm{2}^{\pmb{D_1^{'}}},\pmb{x}\sm{2}^{\pmb{D_1^{''}}}(\pmb{y}\sm{2}-\pmb{\xi}\sm{2})^{\pmb{\Delta}} )$ is a collection of first integrals of ${\omega}\cdot \mathcal{O}_q$. Since $\pmb{D_1^{'}}$ has rank equal to $\pmb{D_1}$, we conclude that:
\[
(\pmb{x}\sm{2}^{\pmb{D_1^{'}}},(\pmb{y}\sm{2}-\pmb{\xi}\sm{2})^{\pmb{\Delta}})
\]
is another collection of first integrals of ${\omega}\cdot \mathcal{O}_q$.\\
\\
\noindent\emph{Third change of coordinates:} Since $\pmb{\Delta}$ is of maximal rank, there exists a coordinate system $(\pmb{x}\sm{3},\pmb{y}\sm{3},\pmb{z}\sm{3},\pmb{w}\sm{3})$ where $\pmb{x}\sm{3} = \pmb{x}\sm{2}$ and $\pmb{w}\sm{3}=\pmb{w}\sm{2}$ such that:
\[
 (\pmb{y}\sm{2}-\pmb{\xi}\sm{2})^{\pmb{\Delta}} = \pmb{y}\sm{3}-\pmb{\xi}\sm{3}
\]
which finally implies that the monomial functions 
\[
(\pmb{x}\sm{3}^{\pmb{D_1^{'}}},\pmb{y}\sm{3})
\]
are fisrt integrals of ${\omega}\cdot \mathcal{O}_q$. This implies that this coordinate system is monomial. Furthermore, since $\pmb{x}\sm{2}^{\pmb{\alpha_1}}$ is independent of the $\pmb{y}\sm{2}$ coordinate, we finally conclude that:
\[
\begin{aligned}
\pmb{u}&=\pmb{x}\sm{3}^{\pmb{A_1}}(\pmb{y}\sm{3}-\pmb{\xi}_1\sm{3})^{\pmb{\Lambda}_1\sm{3}}(\pmb{z}\sm{3}-\pmb{\xi}_2\sm{3})^{\pmb{\Lambda}_2\sm{3}}\\ 
v&=\pmb{x}\sm{3}^{\pmb{\alpha_1}}\\
\pmb{w}&=\pmb{w}\sm{3}
\end{aligned}
\]
where $\pmb{\Lambda}=\begin{bmatrix}
                             \pmb{\Lambda_1} & \pmb{\Lambda_2}
                            \end{bmatrix}
$ is a maximal rank matrix with entries in $\mathcal{K}$ and $\pmb{\xi}\sm{3} = (\pmb{\xi}_1\sm{3},\pmb{\xi}_2\sm{3})$ is a vector where no entry is zero. This proves that the coordinate system is monomial, {and this is the coordinate system of the enunciate of the Lemma}.

{The rest of the argument (in respect to a monomial $\umon{\gamma}$ which is \textbf{not} a first integral of ${\omega}\cdot \mathcal{O}_p$) is exactly the same as in the proof of Lemma \ref{lem:claim1}.}
\end{proof}

\subsection{Proof of Proposition \ref{prop:Dropping}}

By hypothesis, there exists a local coordinate system $(\pmb{u},v,\pmb{w})$ that satisfies the prepared normal form at $p$ with $\nu = \nu(p,{\omega},\mathcal{D})$. In particular, apart from re-indexing, we assume that $T_1$ satisfies the conclusions of proposition \ref{prop:preparation}. Since ${\omega}$ is $\mathcal{K}$-monomial, by Lemma \ref{lem:Fi}, there exists $m-d$ monomials $\umon{B} = (\umon{\beta_1},\dots,\allowbreak \umon{\beta_{m-d}})$, such that
 \[
{\omega}\cdot \mathcal{O}_p =\{\partial \in Der_p(-logE);\text{ } \partial(\umon{\beta_i})\equiv 0 \text{ for all } i\}
\]
Let us now consider the ideal $\mathcal{J}$ generated by:
\[
v^\nu, \text{ and } \{v^j \pmb{u}^{\pmb{r}_{j}}b_{j}\}_{1\leq j<d}, \text{ and } \pmb{u}^{\pmb{\beta}}
\]
where we recall that all $b_{j}$ are either units or zero for $j>0$ and $\pmb{u}^{\pmb{\beta}}$ is in the ideal, provided that $b_{0}\neq 0$ (note that we include only the monomial $\umon{\beta}$ and not $\umon{\beta}w_{{m}}^{\epsilon}$ in the ideal). Now, consider a sequence of blowings-up:
\[
\pmb{\tau}: (M_r,{\omega}_r,\mathcal{D}_r,E_r) \rightarrow (M_0,{\omega}_0,\mathcal{D}_0,E_0)
\]
that principalize $\mathcal{J}$, where $M_0$ is any fixed open neighborhood of $p$ where $\mathcal{J}$ is well-defined. Since $\mathcal{J}$ is generated by monomials in the variables $\pmb{u}$ and $v$, this sequence can be chosen to be combinatorial with respect to the divisor $F:= \{u_1 \cdots u_l \cdot  v =0 \} $ (see Definition \ref{def:combblowingsup}). In particular, the sequence $\tau$ is ${\omega}$-admissible.\\
\\
Now, let $q$ be a point of $M_r$ in the pre-image of $p$. We claim that 
\[
\nu(q,{\omega}_r,\mathcal{D}_r) < \nu(p,{\omega},\mathcal{D})
\]
which finishes the proof of the Proposition. Indeed, since $\pmb{\tau}$ is a sequence of combinatorial blowings-up in respect to the divisor $F$, the point $q$ satisfies the hypothesis of either lemma \ref{lem:claim1} or \ref{lem:claim2}. Thus, we have two cases to consider: \medskip

\noindent\emph{Case 1:} We assume we are in conditions of lemma \ref{lem:claim1}. There exists a monomial system of coordinates $(\pmb{x},\pmb{y},z,\pmb{w}) = (x_1,\ldots, x_t,y_{t+1},\ldots ,\allowbreak y_{r},z,\pmb{w})$ centered at $q$ such that
\begin{equation}
 \label{eq:FC}
\begin{aligned}
\pmb{u}&=\pmb{x}^{\pmb{A}_1}(\pmb{y}-\pmb{\xi})^{\pmb{\Lambda}}\\
v&=\xmon{\alpha_1}(z-\zeta)\\
\pmb{w}&=\pmb{w}
\end{aligned}
\end{equation}
where $\xi_{j} \neq 0$ for all $j$ and the matrix $\pmb{\Lambda}= (\lambda_{i,j})$ of exponents has maximal rank, with $\lambda_{i,j}\in\mathcal{K}$. In particular, $\partial_z$ is contained in $\widetilde{{\omega}}\cdot \mathcal{O}_q$ (this follows from the above coordinate change). So, after blowing-up we have the following expressions:
\begin{equation}
\begin{aligned}
\tau^{\ast}T_1&=U\pmb{x}^{S_\nu}(z-\zeta)^\nu+\sum_{=1}^{\nu-1}\pmb{x}^{S_{j}}(z-\zeta)^ic_{j}(\pmb{x},\pmb{y},\pmb{w})+\pmb{x}^{S_{0}}c_{0}(\pmb{x},\pmb{y},\pmb{w})
\end{aligned}
\label{eq:FCN}
\end{equation}
where:
\begin{itemize}
 \item The function $U$ is a unit of the form $\widetilde{U}(\pmb{x},\pmb{y},\pmb{w}) + \pmb{x}^{\pmb{\alpha_1}} \Omega(\pmb{x},\pmb{y},z,\pmb{w})$, where $\widetilde{U}(\pmb{x},\pmb{y},\pmb{w})$ is a unit and $\pmb{\alpha_1} \neq 0$ (because $q$ is in the pre-image of $p$);
\item For $j>0$ the functions $c_{j}$ are either zero or units (that don't depend on $z$);
\item The term $\pmb{x}^{S_0}c_{0}$ is the pullback of $b_{0}$. In particular, either $c_{0} = 0$ or it is equal to $w_{{m}}^{\epsilon} \widetilde{c}_{0}$ where $\widetilde{c}_{0}$ is a unit.
\end{itemize}

We consider three cases depending on which generator of $\mathcal{I}$ pulls back to be a generator of the pull-back of $\mathcal{I}^*$:\medskip

\noindent{\emph{Case 1.1:}}[The pull back of $v^{\nu}$ generates $\mathcal{J}^{\ast}$, i.e. $S_{\nu}=\min\{S_{\nu},S_{j},S_0\}$]
In this case, by equation \eqref{eq:FCN}, we have:

\[
\begin{aligned}
\tau^{\ast}T_1=\pmb{x}^{S_{\nu}}&\left[\left(\widetilde{U}z+\widetilde{U}\zeta \nu + \pmb{x}^{\pmb{\alpha_1}}\Omega_2\right)z^{\nu-1} + \right.\\
& \left. + \text{ terms where the exponent of }z\text{ is }<\nu-1 \right]
\end{aligned}
\]
where $\pmb{\alpha_1}$ is a non-zero matrix and $\Omega_2 = [z+ \zeta \nu]\Omega$. Since $\widetilde{U}z+\widetilde{U} \nu \zeta + \pmb{x}^{\pmb{\alpha_1}} \Omega_2$ is a unit and the vector-field $\partial_z$ belongs to ${\omega}_r$, we conclude that $\nu(q,{\omega}_r,\mathcal{D}_r) \leq \nu-1$.\medskip

\noindent{\emph{Case 1.2:}}[There is a maximum $0 < j_1< d$ such that the pull back of $\pmb{u}^{\pmb{r}_{j_1}}v^{j_1}$ generates $\mathcal{J}^{\ast}$, i.e $S_{j_1}=\min\{S_{\nu},S_{j},S_0\}$, $S_{\nu}> S_{j_1}$ and $S_{j}> S_{j_1}$ for $j>j_1$]. In this case, by equation \eqref{eq:FCN}, we have:
\[
\tau^{\ast}T_{1}=\xmon{S_{j_1}}\left[(z-\zeta)^{j_1}c_{j_1}+\sum_{j=0}^{j_1-1}\xmon{S_{j}-S_{j_1}}(z-\zeta)^{j}c_{j}+\Omega(\pmb{x},\pmb{y},z,\pmb{w}) \right]
\]
where $\Omega( \pmb{0} ,\pmb{y},z,\pmb{w}) \equiv 0$.  Since $c_{j_1}$ is an unit and the vector-field $\partial_z$ belongs to ${\omega}_r$, we conclude that $\nu(q,{\omega}_r,\mathcal{D}_r) \leq j_1< \nu$.\medskip

\noindent{\emph{Case 1.3:}}[The pull-back of $\umon{\beta}$ is the \textit{only} generator of $\mathcal{J}^{\ast}$ i.e. $S_0=\min\{S_{\nu},S_{j},S_0\}$ and $S_{\nu} >S_0$, $S_{j}>S_0$] 
In this case, we consider two cases depending on $\epsilon$: \medskip

\noindent{\emph{Case 1.3a, $\epsilon=1$:}} Then 
\[
\tau^{\ast}T_{1}=\pmb{x}^{S_0}\left[w_{{m}}\widetilde{c}_{0}+\Omega(\pmb{x},\pmb{y},z,\pmb{w})\right]
\]
where $\widetilde{c}_{0}$ is a unit and $\Omega(\pmb{0},\pmb{y},z,\pmb{w})\equiv0$. Since the vector-field $\partial_{w_{{m}}}$ belongs to ${\omega}_r \cdot \mathcal{O}_q$, we conclude that that $\nu(q,{\omega}_r,\mathcal{D}_r) \leq 1 < \nu$.\medskip

\noindent{\emph{Case 1.3b, $\epsilon=0$:}} In this case, note that the monomial $\umon{\beta+\delta}$ is \textbf{not} a first integral of ${\omega}\cdot \mathcal{O}_p$. Thus, Lemma \ref{lem:claim1} guarantees that the total transform $\umon{\beta + \delta} = \pmb{x}^{S_0 + \widetilde{\delta}}\widetilde{W} $, where $\widetilde{W} = (\pmb{y}-\pmb{\gamma})^{\pmb{(\delta+\beta) \Lambda}}$ is a unit, satisfies one of the following:
\begin{itemize}
 \item Either $\pmb{x}^{S_0 + \pmb{\widetilde{\delta}}}$ is \textbf{not} a first integral of $\widetilde{{\omega}}\cdot \mathcal{O}_q$, which implies that
 \[
\tau^{\ast}[\umon{\delta} T_{1}]=\pmb{x}^{S_0+\pmb{\widetilde{\delta}}}U
\]
for some unit $U$. We conclude that that $\nu(q,{\omega}_r,\mathcal{D}_r) = 0 < \nu$;
\item Or, there exists a regular vector-field $\partial_{y_i} \in \widetilde{{\omega}}\cdot \mathcal{O}_q$ such that $\partial_{y_i}\widetilde{W}$ is a unit. In particular:
\[
\tau^{\ast}[\umon{\delta} T_{1}]=   \pmb{x}^{S_0+\pmb{\widetilde{\delta}}}\widetilde{W}(0)+  \pmb{x}^{S_0+\pmb{\widetilde{\delta}}}\left[\widetilde{W}-\widetilde{W}(0) + \Omega(\pmb{x},\pmb{y},z,\pmb{w})\right]
\]
where $\Omega( \pmb{0} ,\pmb{y},z,\pmb{w}) \equiv 0$ and the monomial $\pmb{x}^{S_0+\pmb{\widetilde{\delta}}}W(0)$ is a first integral of $\widetilde{{\omega}}\cdot \mathcal{O}_q$. We conclude that that $\nu(q,{\omega}_r,\mathcal{D}_r) \leq 1 < \nu$.
\end{itemize}
\medskip

\noindent\emph{Case 2:} We assume we are in conditions of Lemma \ref{lem:claim2}. There exists a monomial system of coordinates $(\pmb{x},\pmb{y},\pmb{w}) = (x_1,\ldots, x_t,y_{t+1},\ldots ,\allowbreak y_{r+1},\pmb{w})$ centered at $q$ such that
\begin{equation}
\label{eq:SC}
\begin{aligned}
\pmb{u}&=\pmb{x}^{A_1}(\pmb{y}-\pmb{\xi})^{\Lambda}\\
v&=\xmon{\alpha_1}\\
\pmb{w}&=\pmb{w}
\end{aligned}
\end{equation}
 
where $\widetilde{\xi_{j}} \neq 0$ for all $j$ and the matrix $\pmb{\Lambda}= (\lambda_{i,j})$ of exponents has maximal rank and $\pmb{\alpha_1}$ doesn't belong to the span of the rows of $\pmb{A_1}$. So, after blowing-up we have the following expression:
\begin{equation}
\begin{aligned}
\tau^{\ast}T_1&=U\pmb{x}^{S_\nu}+\sum_{=1}^{\nu-1}\pmb{x}^{S_{j}}c_{j}(\pmb{x},\pmb{y},\pmb{w})+\pmb{x}^{S_{0}}c_{0}(\pmb{x},\pmb{y},\pmb{w})
\end{aligned}
\label{eq:SCN}
\end{equation}
where:
\begin{itemize}
 \item The function $U$ is a unit and, for $j>0$, the functions $c_{j}$ are either zero or units (that don't depend on $z$);
 \item The term $\pmb{x}^{S_0}c_{0}$ is the pullback of $b_{0}$. In particular, either $c_{0} = 0$ or it is equal to $w_{{m}}^{\epsilon}\widetilde{c}_{0}$ where $\widetilde{c}_{0}$ is a unit;
 \item We remark that:
\begin{align*}
S_\nu&= \nu \pmb{\alpha}_1\\
S_{j}&= j\pmb{\alpha}_1+\pmb{r}_{1,j}\pmb{A}_1\text{, for }j=0,\ldots,\nu-2.
\end{align*}
So each $S_\nu$ and $S_{j}$ is a sum of an element of the span of the rows of $\pmb{A}_1$ and a different multiple of the $\pmb{\alpha}_1$. Since $\pmb{\alpha}_1$ is linearly independent with the rows of $\pmb{A}_1$, this means that the exponents $S_\nu$ and $S_{j}$ are all distinct. Therefore all of the multi-indexes $S_{j}$ must be different.
\end{itemize}

We consider three cases depending on which generator of $\mathcal{I}$ pulls back to be a generator of the pull-back of $\mathcal{I}^*$:\medskip

\noindent{\emph{Case 2.1:}}[The pull back of $v^{\nu}$ generates $\mathcal{J}^{\ast}$, i.e. $S_{\nu}=\min\{S_{\nu},S_{j},S_0\}$]
In this case, from equation \eqref{eq:SCN}, we have:
\[
 \tau^{\ast}[\umon{\delta}T_{1}]=\xmon{S_{\nu} + \pmb{\delta A}_1}\left[\widetilde{U}+\Omega(\pmb{x},\pmb{y},z,\pmb{w}) \right]
\]
where $\widetilde{U} = U (\pmb{y}- \pmb{\xi})^{\pmb{\delta \Lambda}}$ is a unit and $\Omega( \pmb{0} ,\pmb{y},z,\pmb{w}) \equiv 0$. Since $\xmon{S_{\nu}+\pmb{\delta A}_1}$ is not a first integral of $\widetilde{{\omega}}$ (which follows from the fact that $\pmb{\alpha}_1$ doesn't belong to the span of the rows of $\pmb{A_1}$), we conclude that $\nu(q,{\omega}_r,\mathcal{D}_r) =0< \nu$.
\medskip

\noindent{\emph{Case 2.2:}}[There is a maximum $0 < j_1< \nu$ such that the pull back of $\pmb{u}^{\pmb{r}_{j_1}}v^{j_1}$ is a generator of $\mathcal{J}^{\ast}$, i.e $S_{j_1}=\min\{S_{\nu},S_{j},S_0\}$]. In this case, from equation \eqref{eq:SCN}, we have:
\[
\tau^{\ast}[\umon{\delta}T_{1}]=\xmon{S_{j_1} + \pmb{\delta A}_1}\left[\widetilde{c}_{j_1}+\Omega(\pmb{x},\pmb{y},z,\pmb{w}) \right]
\]
where $\widetilde{c}_{j_1} = c_{j_1} (\pmb{y}- \pmb{\xi})^{\pmb{\delta \Lambda}}$ is a unit and $\Omega( \pmb{0} ,\pmb{y},z,\pmb{w}) \equiv 0$. Since $\xmon{S_{j_1}+\widetilde{\delta}}$ is not a first integral of $\widetilde{{\omega}}$ (which follows from the fact that $\pmb{\alpha}_1$ doesn't belong to the span of the rows of $\pmb{A_1}$), we conclude that $\nu(q,{\omega}_r,\mathcal{D}_r) =0< \nu$.
\medskip

\noindent{\emph{Case 2.3:}}[The pull-back of $\umon{\beta}$ is the generator of $\mathcal{J}^{\ast}$ i.e. $S_0=\min\{S_{\nu},S_{j},S_0\}$]
In this case, we consider two cases depending on $\epsilon$: \medskip

\noindent{\emph{Case 2.3a, $\epsilon=1$:}} Then 
\[
\tau^{\ast}T_{1}=\pmb{x}^{S_0}\left[w_{{m}} \widetilde{c}_0+\Omega(\pmb{x},\pmb{y},z,\pmb{w})\right]
\]
where $\widetilde{c}_0$ is a unit and $\Omega(\pmb{0},\pmb{y},z,\pmb{w})\equiv0$. Since the vector-field $\partial_{w_{{m}}}$ belongs to ${\omega}_r \cdot \mathcal{O}_q$, we conclude that that $\nu(q,{\omega}_r,\mathcal{D}_r) \leq 1 < \nu$.\medskip

\noindent{\emph{Case 2.3b, $\epsilon=0$:}} In this case, note that the monomial $\umon{\beta+\delta}$ is \textbf{not} a first integral of ${\omega}\cdot \mathcal{O}_p$. Thus, Lemma \ref{lem:claim2} guarantees that the total transform $\umon{\beta + \delta} = \pmb{x}^{S_0 + \delta A_1}\widetilde{W} $, where $\widetilde{W} = (\pmb{y}-\pmb{\gamma})^{\pmb{(\delta +\beta)\Lambda}}$ is a unit, satisfies one of the following conditions:
\begin{itemize}
 \item Either $\pmb{x}^{S_0 + \pmb{\delta A_1}}$ is \textbf{not} a first integral of $\widetilde{{\omega}}\cdot \mathcal{O}_q$, which implies that
 \[
\tau^{\ast}[\umon{\delta} T_{1}]=\pmb{x}^{S_0+\pmb{\delta A_1}}U
\]
for some unit $U$. We conclude that that $\nu(q,{\omega}_r,\mathcal{D}_r) = 0 < \nu$;
\item Or, there exists a regular vector-field $\partial_{y_i} \in \widetilde{{\omega}}\cdot \mathcal{O}_q$ such that $\partial_{y_i}\widetilde{W}$ is a unit. In particular:
\[
\tau^{\ast}[\umon{\delta} T_{1}]=   \pmb{x}^{S_0+\pmb{\delta A_1}}\widetilde{W}(0)+  \pmb{x}^{S_0+\pmb{\delta A_1}}\left[\widetilde{W}-\widetilde{W}(0) + \Omega(\pmb{x},\pmb{y},z,\pmb{w})\right]
\]
where $\Omega( \pmb{0} ,\pmb{y},z,\pmb{w}) \equiv 0$ and the monomial $\pmb{x}^{S_0+\pmb{\delta A_1}}W(0)$ is a first integral of $\widetilde{{\omega}}\cdot \mathcal{O}_q$. We conclude that that $\nu(q,{\omega}_r,\mathcal{D}_r) \leq 1 < \nu$.
\end{itemize}

\section*{Acknowledgments}

{I would like to express my gratitude to Edward Bierstone and Franklin Vera-Pacheco for the useful discussions. I would also like to thank Felipe Cano, Wanderson Costa e Silva, Pavao Marde\v{s}i\'{c} and Daniel Panazzolo for the useful comments. Finally, I thank the anonymous referee for several useful suggestions and comments which improved the quality of the paper}.

\bibliographystyle{alpha}

\end{document}